\documentclass[a4paper,12pt]{article}
\usepackage[utf8]{inputenc}
\usepackage{tikz}
\usetikzlibrary{angles,quotes}
\usepackage{float}
\usepackage{extarrows}
\usepackage{graphicx}

\usepackage{enumerate}
\usepackage{caption}
\usepackage{subcaption}
\usepackage{graphicx}
\usepackage{epstopdf}
\usepackage{enumerate}
\usepackage{multicol}
\usepackage{color}
\usepackage{amsmath,amssymb,amscd}
\usepackage{amsthm}
\usepackage{hyperref}           
\usepackage{cleveref}
\usepackage{cite}

\usepackage{pdfpages}
\usepackage{hyperref}
\usepackage{url}

\newcommand{\RR}{\mathbb{R}}

\newcommand{\perim}{\mathrm{per}}
\newcommand{\diam}{\mathrm{diam}}
\newcommand{\area}{\mathrm{area}}

\newcommand{\round}[1]{ \ensuremath{\left( #1 \right)}  }

\newcommand{\per}[1]{ \ensuremath{\perim\left( #1 \right)}  }
\newcommand{\ar}[1]{ \ensuremath{\area\left( #1 \right)}  }

\theoremstyle{plain}
\newtheorem{thm}{Theorem}[section]
\newtheorem{lem}[thm]{Lemma}

\newtheorem{cor}[thm]{Corollary}
\newtheorem{prop}[thm]{Proposition}
\newtheorem*{prop*}{Proposition}
\newtheorem{claim}[thm]{Claim}

\theoremstyle{definition}

\newtheorem{ex}[thm]{Example}

\theoremstyle{remark}
\newtheorem{rem}{Remark}

\newcommand{\ga}{\gamma}
\newcommand{\sa}{\sphericalangle}

\title{Optimal embedded and enclosing isosceles triangles}
\author{Áron Ambrus, Mónika Csikós\thanks{Universit\'e Paris Cit\'e and ENS Paris. Research partially supported by the grant ANR-19-CE48-0016 from the French National Research Agency (ANR).}, Gergely Kiss\thanks{R\'enyi Institute, Budapest. Research partially supported by Premium Postdoctoral Fellowship of the Hungarian Academy of Sciences and by National Research, Development and Innovation Office (NKFIH) grant K-124749.}, János Pach\thanks{R\'enyi Institute, Budapest and IST Austria. Research partially supported by National Research, Development and Innovation Office (NKFIH) grant K-131529 and ERC Advanced Grant ``GeoScape.''}, Gábor Somlai\thanks{Eötvös Loránd University, R\'enyi Institute, partially supported by the J\'anos Bolyai Research Fellowship and by the Thematic Excellence Programme TKP2020-NKA-06.}}
\date{}

\begin{document}

\maketitle

\begin{abstract}
    Given a triangle $\Delta$, we study the problem of determining the smallest enclosing and largest embedded isosceles triangles of $\Delta$ with respect to area and perimeter.  This problem was initially posed by Nandakumar~\cite{N, NandakumarBlog} and was first studied by Kiss, Pach, and Somlai \cite{KPS2020}, who showed that if $\Delta'$ is the smallest area isosceles triangle containing $\Delta$, then $\Delta'$ and $\Delta$ share a side and an angle. 
    In the present paper, we prove that for any triangle $\Delta$, every maximum area isosceles triangle embedded in  $\Delta$ and  every maximum perimeter isosceles triangle embedded in $\Delta$ shares a side and an angle with $\Delta$. Somewhat surprisingly, the case of minimum perimeter enclosing triangles is different: there are infinite families of triangles $\Delta$ whose minimum perimeter isosceles containers do not share a side and an angle with $\Delta$.
\end{abstract}

\section{Introduction}
The following classical problem is the starting point of our investigation. Given two convex bodies, $C$ and $C'$ in $\RR^d$, decide whether $C$ can be moved into a position where it covers $C'$. One can easily list some necessary conditions, for instance, the volume, the surface area and the diameter of $C$ has to be at least as large as the one of $C'$. However, solving the decision problem can be rather challenging, even in $\RR^2$, or for special cases that might seem friendly at first sight.

For instance, consider the setup,  where $C'$ is the `shadow' of $C$, that is, $C$ is embedded into $\RR^3$ and $C'$ is the orthogonal projection of $C$ onto a $2$-dimensional affine subspace. 
The necessary conditions are clearly satisfied and it looks plausible that there is always a congruent copy of $C$ which covers $C'$. However, the proof of this fact is far from straightforward~\cite{DeM86,KoT90}, and rather surprisingly, the result does not generalize to higher dimensions: for $d \geq 3$, no convex $d$-polytope embedded in $\RR^{d+1}$ can cover all of its shadows \cite{DeM86}. 

Another special case is where both convex bodies are triangles in $\RR^2$: given two triangles $\Delta$ and $\Delta'$, find an efficient way to decide whether $\Delta$ can be brought into a position where it covers $\Delta'$. This is a classical problem posed by Steinhaus~\cite{St64} in $1964$ and an algorithmic solution was proposed only $29$ years later by Post \cite{Post}, who described a set of $18$ polynomial inequalities of degree $4$ 
such that a copy of $\Delta$ can cover $\Delta'$ if and only if at least one of these inequalities are satisfied.
The key geometric component of Post's solution is the following.
\begin{lem}[{Post~\cite{Post}}]\label{tPost}
If a triangle $\Delta$ can be moved to a position where it covers another triangle $\Delta'$, then one can also find a covering position of $\Delta$ with a side that contains one side of $\Delta'$.
\end{lem}
Results of this kind help us to reduce the number of configurations to consider, and are of both theoretical and practical interest.
\subsection{Optimal covers from a class}  
In the present paper, we study a variant of the covering problem where the body $C$ (or $C'$) is not fixed, but can be chosen from a family of possible objects and we want to find a solution which is in some sense optimal, for example, has minimum area or perimeter.

Several classical problems in geometry can be viewed as covering problems of this kind: finding an optimal enclosing triangle, polygon, or
ellipse (L\"owner-John ellipse) for a given input set  \cite{BC,BoD,D,E,Fe,I,JW,R1,R2} as well as their higher dimensional analogues (that is, simplices, polytopes, ellipsoids,  \cite{OR, K, We}). Apart from their theoretical interest, these problems have found applications in various areas of computer science and mathematics (optimization,
packing and covering, approximation algorithms, convexity, computational
geometry), see \cite{GLS, KY, Sc}. 
In the past decade, several explicit algorithms were proposed for the case of triangles \cite{BC, parvu, van, se}.\\
\smallskip

\noindent
Nandakumar~\cite{N, N2, NandakumarBlog}
 raised the following two special instances of the above question: {\em given a triangle $\Delta$, determine the minimum area and the minimum perimeter isosceles
triangles that contain $\Delta$.} In what follows, we answer these questions, together with their `dual' versions:
 \emph{given a triangle $\Delta$, determine the maximum area and the maximum perimeter isosceles
triangles embedded (i.e., contained) in $\Delta$.}\\

\noindent
The case of minimum area isosceles containers has been recently studied by Kiss, Pach, and Somlai \cite{KPS2020}: they described all isosceles containers of a given triangle $\Delta$ for which the minimum is attained. Here, we complete the picture: we characterize the optimal solutions of the other three problems stated above. 
We will conclude that for three of the above problems, the optimum is attained for a `trivial' configuration, where the two triangles share a side and an angle at one end of this side.

\begin{thm} \label{tfo}
    Let $\Delta$ be a triangle in $\RR^2$ and~\vspace{-.5em}
\begin{enumerate}[(i)]
\setlength{\itemsep}{1pt}
  \setlength{\parskip}{1pt}
      \item \label{mainthm:i}let $\Delta''\supseteq \Delta$ be a \emph{minimum area isosceles container} of $\Delta$. 
      Then $\Delta''$ and $\Delta$ have a side in common and at one endpoint of this side they also have the same angle \cite{KPS2020};
        \item\label{mainthm:ii} let $\Delta'\subseteq \Delta$ be a \emph{maximum area embedded isosceles triangle} in $\Delta$. Then 
$\Delta'$ and $\Delta$ have a side in common and at one endpoint of this side they also have the same angle;
\item\label{mainthm:iii} let $\Delta'\subseteq \Delta$ be a \emph{maximum perimeter embedded isosceles triangle} in $\Delta$. 
Then $\Delta'$ and $\Delta$ have a side in common and at one endpoint of this side they also have the same angle.
\end{enumerate}
\end{thm}
\noindent
Somewhat surprisingly, the analogous statement is false for minimum perimeter containers.
\begin{thm}\label{thm:main-counterexamples}
There are infinite families of triangles $\Delta$ such that none of their minimum perimeter isosceles containers shares a side with $\Delta$ (and an angle at the end of this side). 

We describe $5$ different types of isosceles containers such that any triangle $\Delta$ has a minimum perimeter isosceles container $\Delta'$ belonging to one of these types.
Only $3$ out of these types will have the property that $\Delta$ and $\Delta'$ share a side and at one of the endpoints of this side they also have the same angle.
\end{thm}
Our paper is organized as follows. In \Cref{sec:prelim}, we fix the notation and list some easy preliminary statements. In
\Cref{sec:max-area} and \Cref{sec:max-perim}, we present the proofs of \Cref{tfo}(ii) and  \Cref{tfo}(iii), respectively. Finally, \Cref{sec:min-perim} is dedicated to the description of the $5$ types of isosceles containers mentioned in \Cref{thm:main-counterexamples} and to the proof of this result. \\

We are grateful for Ilya I. Bogdanov (MIPT) for his valuable remarks.

\section{Preliminaries and notation}\label{sec:prelim}
In this section, we introduce the notation used in this note and state three easy lemmas (Lemmas \ref{lem:01}-\ref{lem:ie2} and \ref{lem:ie1}). Their straightforward proofs are given in the Appendix.
\begin{lem}\label{lem:01}
Let $\Delta_1$ and $\Delta_2$ be two triangles.
\begin{enumerate}[(i)]
\item 
Any maximum area (resp. perimeter) similar copy $\Delta_1'\subseteq \Delta_2$ of $\Delta_1$ satisfies the following properties: ~\vspace{-.5em}
\begin{enumerate}[(a)]
 \setlength{\itemsep}{1pt}
  \setlength{\parskip}{1pt}
\item there is a side of $\Delta_2$ that contains a side of $\Delta_1'$;
\item every side of $\Delta_2$ contains a vertex of $\Delta_1'$;
\item $\Delta_1'$ and $\Delta_2$ have a common vertex.
\end{enumerate}
\item
Any minimum area (resp. perimeter) similar copy $\Delta_1' \supseteq\Delta_2$ of $\Delta_1$ satisfies the following properties: ~\vspace{-.5em}
\begin{enumerate}[(a)]
 \setlength{\itemsep}{1pt}
  \setlength{\parskip}{1pt}
\item there is a side of $\Delta_1'$ that contains two vertices of $\Delta_2$;
\item every side of $\Delta_1'$ contains a vertex of $\Delta_2$;
\item $\Delta_1'$ and $\Delta_2$ have a common vertex.
\end{enumerate}
\end{enumerate}
\end{lem}
\noindent
Optimal isosceles enclosing and embedded triangles satisfy further properties.   
\begin{lem}\label{lem:m1} ~\vspace{-.5em}
\begin{enumerate}[(i)]
 \setlength{\itemsep}{1pt}
  \setlength{\parskip}{1pt}
  \item For every triangle $\Delta$ there exists a minimum area (resp. perimeter) isosceles container of $\Delta$, and a maximum area (resp. perimeter) isosceles triangle embedded in $\Delta$.
  \item If $\Delta_1$ is a maximum area (resp. perimeter) isosceles triangle embedded in $\Delta$, then every vertex of $\Delta_1$ lies on a side of $\Delta$.
  \item\label{lemitem:242} If $\Delta_2$ is a minimum area (resp. perimeter) isosceles container of $\Delta$, then every vertex of $\Delta$ lies on a side of $\Delta_2$.
  \end{enumerate}
\end{lem}
For any two points, $A$ and $B$, let $AB$ denote the closed segment connecting them, and let $|AB|$ stand for the length of $AB$. To unify the presentation, in the sequel we fix a triangle $ABC$ with side lengths $a=|BC|$, $b=|AC|$, $c=|AB|$. If two sides are of the same length, then $ABC$ is the unique minimum area and perimeter isosceles container (and also maximum area and perimeter embedded isosceles triangle) of itself. Therefore without loss of generality, we assume that $a<b<c$.
\subsection{Special embedded isosceles triangles}\label{sec:special-embedded}
Given a triangle $ABC$, we describe its special embedded isosceles triangles, that is, all those isosceles triangles contained in $ABC$ that have a common side with $ABC$ and share an angle at one of the endpoints of the common side. Recall that these triangles play a distinguished role in \Cref{tfo}.

\paragraph{Special embedded triangles of the first kind.} Let $A'$ be a point of $AC$ with $|A'C|=|BC|$ and let $B'$ and $A''$ be two points of $AB$ such that $|AB'|=|AC|$ and $|A''B|=|BC|$ (see \Cref{f1}). We say that $A'BC$, $AB'C$, and $A''BC$ are the {\it special embedded triangles of the first kind} associated with $ABC$.
\begin{figure}[H]
    \centering
    \includegraphics[width=\textwidth]{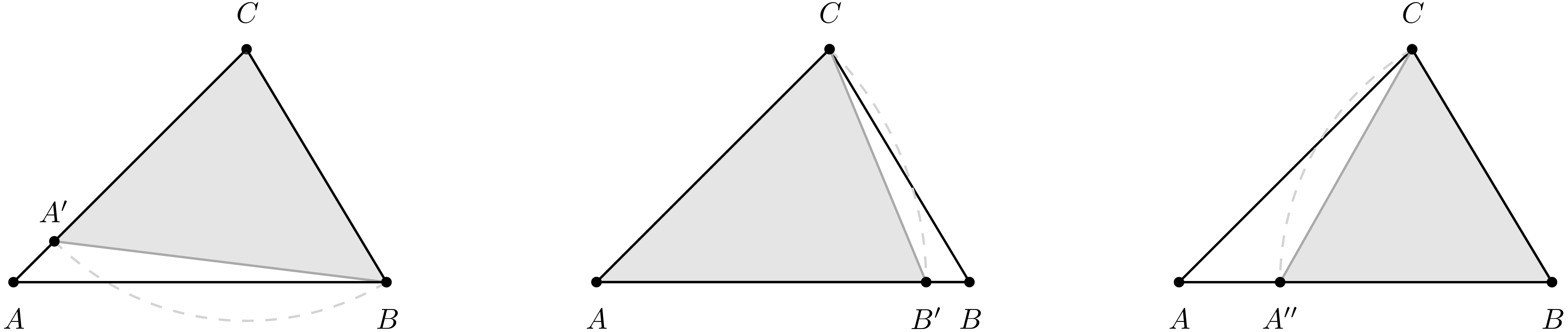}
    \caption{Special embedded triangles of the first kind.}
    \label{f1}
\end{figure}
\paragraph{Special embedded triangles of the second kind.} Let $C_1$ be the intersection of the perpendicular bisector of $AB$ and the segment $AC$. Analogously, let $A_1$ be the intersection of the perpendicular bisector of $BC$ and $AC$, and let $B_1$ be the intersection of the perpendicular bisector of $BC$ and the line $AC$  (see \Cref{f2}). The triangles $A_1BC$, $AB_1C$, and $ABC_1$ are the {\it special embedded triangles of the second kind} associated with $ABC$. 
\begin{figure}[H]
\begin{center}
 \includegraphics[width=\textwidth]{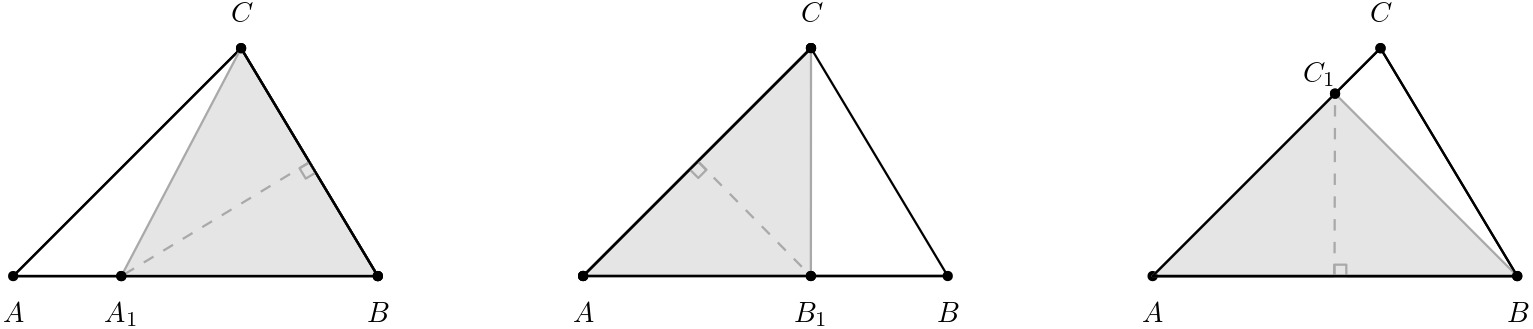}
\end{center}
 \caption{Special embedded triangles of the second kind.}
  \label{f2}
\end{figure}
\paragraph{Special embedded triangles of the third kind.} Let $\overline{A}$ be a point of  $AB$, where $|\overline{A}C|=|BC|$. Analogously, let $\overline{\overline{A}} \in AC$, and $\overline{B} \in BC$ such that $|\overline{\overline{A}}B|=|BC|$, and $|\overline{B}A|=|AC|$ (see \Cref{f3}). Note that if $ABC$ is non-acute, then $\overline{\overline{A}}BC$ and $A\overline{B}C$ do not exist.  $\overline{A}BC$, $\overline{\overline{A}}BC$, and $A\overline{B}C$  are called the {\it special embedded triangles of the third kind} associated with $ABC$. 

\begin{figure}[H]
\begin{center}
 \includegraphics[width=\textwidth]{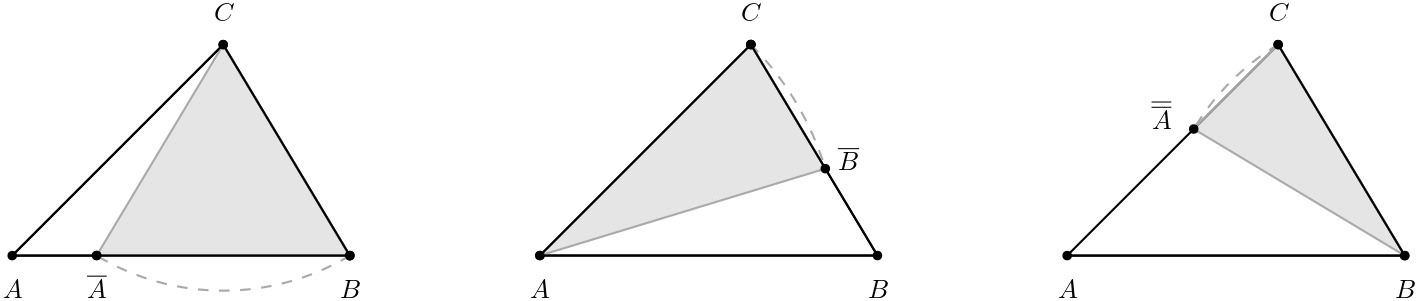}
\end{center}
 \caption{Special embedded triangles of the third kind.}
 \label{f3}
\end{figure}
\subsubsection{Basic inequalities for special embedded triangles.}
We collect a few inequalities on the area and perimeter of special isosceles embedded triangles. 
For a triangle $\Delta$, let  $\perim(\Delta)$ and $\area(\Delta)$ denote the perimeter and the area of ${\Delta}$, respectively.
\begin{lem}\label{lem:ie2}If $ABC$ satisfies $a<b<c$, then\vspace{-.5em}
\begin{enumerate}[(i)]
 \setlength{\itemsep}{1pt}
  \setlength{\parskip}{1pt}
\item $\ar{A''BC}< \ar{A'BC}$;
\item $\ar{A_1BC}<\ar{AB'C}$ and $\ar{AB_1C}<\ar{ABC_1}$;
\item $\ar{\overline{A}BC}<\ar{ABC_1}$, $\ar{\overline{\overline{A}}BC}<\ar{A\overline{B}C}$, \newline and $\ar{A\overline{B}C}<\ar{AB'C}$;
\item if $ABC$ is obtuse, then  $\ar{A'BC}<\ar{ABC_1}$.
\end{enumerate}
\end{lem}
\noindent
\Cref{lem:ie2} imply that only 3 of the special embedded triangles of $ABC$ can be optimal.
\begin{cor}\label{cor:embarea}
If $ABC$ satisfies $a<b<c$, then any maximum area special embedded triangle of $ABC$ is one of the following triangles: $A'BC$, $AB'C$, $ABC_1$.
\end{cor}
\noindent 
We note that similar results hold for the perimeter function implying that any maximum perimeter special embedded triangle of  $ABC$ is $AB'C$, $A_1BC$, or $ABC_1$.
\subsection{Special enclosing isosceles triangles}
Given a triangle $ABC$, now we describe its special enclosing isosceles triangles, that is, all those isosceles triangles containing $ABC$ that have a common side with $ABC$ and share an angle at one of the endpoints of the common side. Recall that these triangles play a distinguished role in Theorems \ref{tfo} and \ref{thm:main-counterexamples}.
\paragraph{Special containers of the first kind.} Let $B'$ denote the point on the ray $\vec{CB}$, for which $|B'C|=|AC|$. Analogously, let $C'$ (and $C''$) denote the
points on $\vec{AC}$ (resp., $\vec{BC}$) such that $|AC'|=|AB|$ (resp., $|BC''|=|AB|$), see \Cref{fig2}.  We call the triangles $AB'C$, $ABC'$, and $ABC''$ {\it special containers of the first kind} associated with $ABC$.
\begin{figure}[!ht]
\begin{center}
\includegraphics[width=0.3\textwidth]{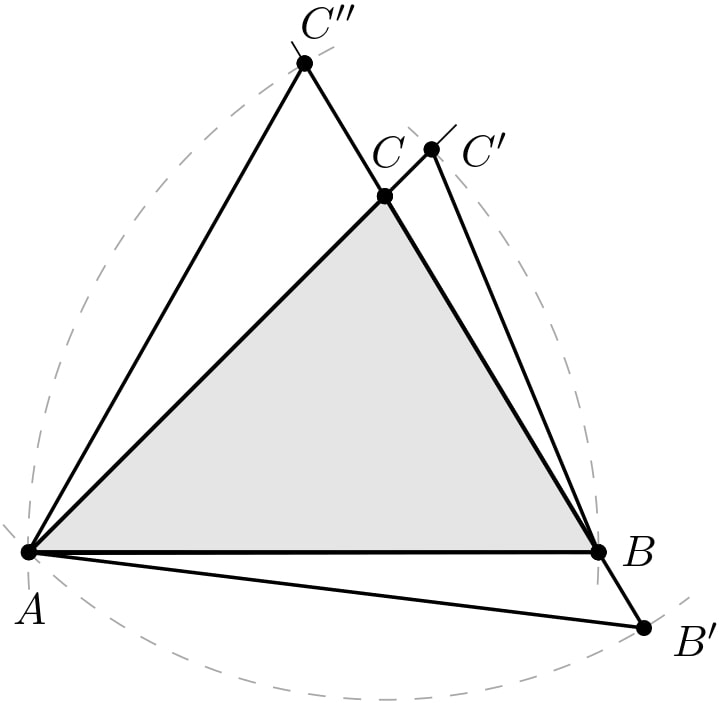}
 \hspace{2em}
 \includegraphics[width=0.32\textwidth]{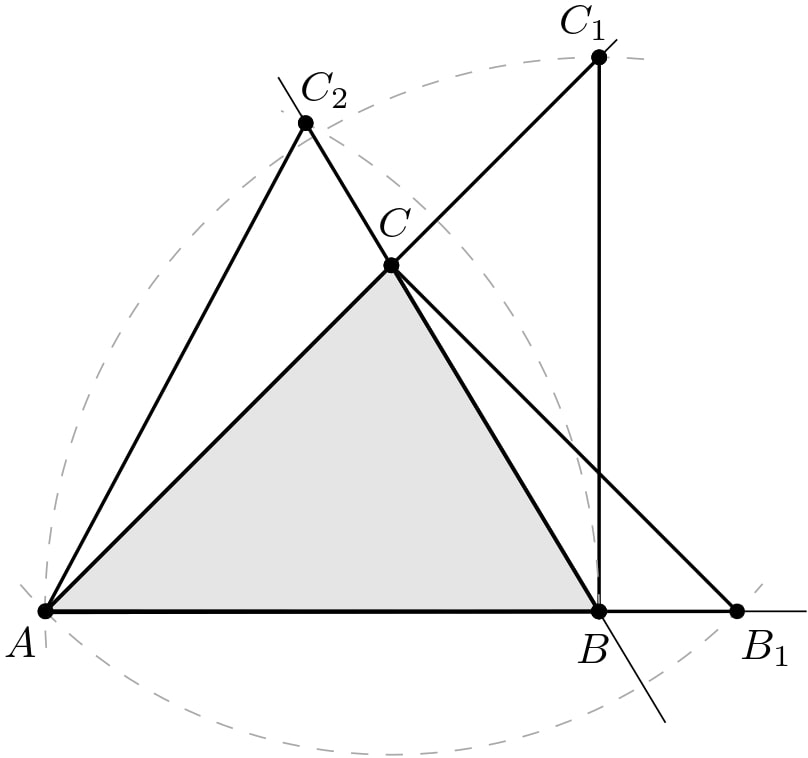}
 
\end{center}
  \caption{Special containers of the first kind ($AB'C, ABC'$, and $ABC''$) and second kind ($AB_1C$, $ABC_1$, and $ABC_2$).}
  \label{fig2}
\end{figure}
\paragraph{Special containers of the second kind.} Let $B_1$ denote the point on the ray $\vec{AB}$, different from $A$, for which $|B_1C|=|AC|$. Analogously, let $C_1$ (resp., $C_2$) denote the point on $\vec{AC}$ (resp., $\vec{BC}$) for which $|BC_1|=|AB|$ and $C_1\neq A$ (resp., $|AC_2|=|AB|$ and $C_2\neq B$), see \Cref{fig2}. The triangles $AB_1C$, $ABC_1$, and $ABC_2$ are called the {\it special containers of the second kind} associated with $ABC$.
\smallskip
\paragraph{Special containers of the third kind.} Let $\overline{A}$ be the intersection of the perpendicular bisector of $BC$ and the line $AC$. Since we have $b=|AC|<|AB|=c$, the point
$\overline{A}$ lies outside of $ABC$. Analogously, denote by $\overline{B}$ (resp., $\overline{C}$) the intersection of the perpendicular bisector of $AC$ (resp. $AB$) and the line $BC$. (If $ABC$ is non-acute $\overline{A}BC$ and $A\overline{B}C$ do not contain $ABC$ (\Cref{fig2.2}).) 
The triangles $\overline{A}BC$, $A\overline{B}C$, and $AB\overline{C}$ are called the {\it special
containers of the third kind} associated with $ABC$, provided that they contain $ABC$.
\begin{figure}[!ht]
\begin{center}
\includegraphics[width=0.6\textwidth]{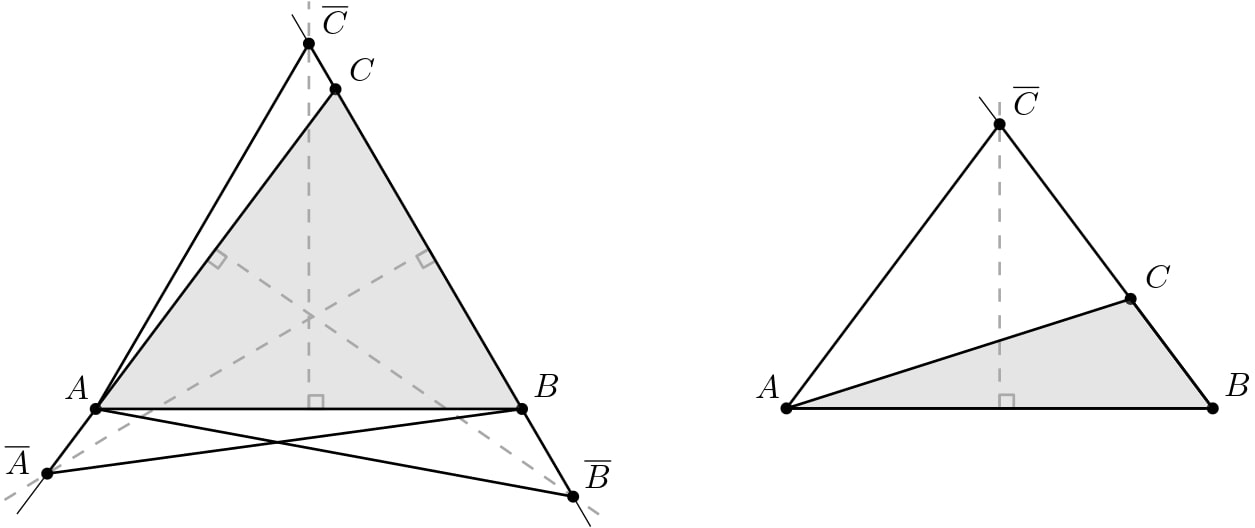}
\end{center}
  \caption{Special containers of the third kind ($\overline{A}BC, A\overline{B}C, AB\overline{C}$) in the acute and in the non-acute cases.}
  \label{fig2.2}
\end{figure}
\subsubsection{Basic inequalities for special containers}
Similarly to the case of maximum area embedded triangles, we can show that not all special containers can be of minimum perimeter.
\begin{lem}\label{lem:ie1}If $ABC$ satisfies $a<b<c$, then~\vspace{-.5em}
\begin{enumerate}[(i)]
 \setlength{\itemsep}{1pt}
 \setlength{\parskip}{1pt}
  \item $\per{ABC'}<\per{ABC''}$ and $\per{AB'C}<\per{AB_1C}$;
  \item $\per{ABC'}<\per{ABC_2} <\per{ABC_1}$; 
  \item $\per{ABC'}<\per{\overline{A}BC}<\per{A\overline{B}C}$.
\end{enumerate}
\end{lem}
\noindent
\Cref{lem:ie1} immediately gives the following corollary.
\begin{cor}\label{cspec}
If $ABC$ satisfies $a<b<c$, then any minimum perimeter special container of $ABC$ is one of the following triangles: $AB'C$, $ABC'$, $AB\overline{C}$. 
\end{cor}
\noindent
Again, we note that similar results hold for the area function implying that a minimum area special container of $ABC$ is 
$AB'C$,  $ABC'$, or $AB_1C$.
\section{Maximum area embedded isosceles triangles \\ -- Proof of \Cref{tfo}(ii)}\label{sec:max-area}
Let $ABC$ be a triangle and let $XYZ$ denote one of its maximum area isosceles embedded triangles. In this section, we prove that $XYZ$ has to be a special embedded triangle. 
We use the notation  $a=|BC|$, $b=|AC|$, $c=|AB|$, $x=|YZ|$, $y = |XZ|$, $z=|XY|$, and assume (with no loss of generality) that $a<b<c$.
By Lemmas \ref{lem:01} and \ref{lem:m1}, we have the following statements on maximum area embedded isosceles triangles. 
 \begin{lem}\label{lemma:Aron}
 Let $XYZ$ be any maximum area isosceles triangle embedded in $ABC$. Then~\vspace{-.5em}
\begin{enumerate}[(i)]
 \setlength{\itemsep}{1pt}
 \setlength{\parskip}{1pt}
     \item a side of $ABC$ contains a side of $XYZ$;
     \item every side of $ABC$ contains a vertex of $XYZ$;
     \item $ABC$ and $XYZ$ have a common vertex;
     \item no vertex of $XYZ$ lies in the interior of $ABC$.
 \end{enumerate}
 \end{lem}
If $XYZ$ has at least two common vertices with $ABC$, then by \Cref{lemma:Aron}(iv), $XYZ$ and $ABC$ have a common side and a common angle. Therefore, we can assume that $ABC$ and $XYZ$ have exactly one common vertex.

Denote the midpoints of the sides $BC$, $AC$, and $AB$ by $m_{A}$, $m_{B}$, and $m_{C}$, respectively. We divide the boundary of $ABC$ into $3$ polylines defined as $$ \widehat{m_Am_B}= m_AC \cup Cm_B, ~\ \widehat{m_Bm_C}=m_BA \cup A m_C, ~\ \widehat{m_Cm_A}= m_CB \cup Bm_A.$$ 
We get the following constraint on the position of  $X$, $Y$, and $Z$:
\begin{lem}\label{clm2}
Let $XYZ$ be a maximum area embedded isosceles triangle of the triangle $ABC$. Then each of $\widehat{m_Am_B}, \widehat{m_Bm_C}$, and $\widehat{m_Cm_A}$ contains exactly one vertex of $XYZ$.
\end{lem}
\begin{proof}
By \Cref{lemma:Aron}, $X,Y,Z$ lies on the boundary of $ABC$. 
Assume, without loss of generality, that $\widehat{m_Am_C}$ contains $X$ and $Z$ (see \Cref{f4}).
\begin{figure}[ht!]
\begin{center}
\begin{tikzpicture}[scale=0.8]
\clip(-1,-0.6) rectangle (8.54,5);
\fill[line width=2pt,fill=black,fill opacity=0.1] (6,0) -- (7.25,1) -- (4.28,3.43) -- cycle;
\draw [line width=1pt] (0,0)-- (8,0);
\draw [line width=1pt] (5,4)-- (0,0);
\draw [line width=1pt] (5,4)-- (8,0);
\draw [line width=1pt] (6,0)-- (7.25,1);
\draw [line width=1pt] (7.25,1)-- (4.28,3.43);
\draw [line width=1pt] (4.28,3.43)-- (6,0);

\draw[dashed,line width=1pt, color=gray, opacity=0.5] (2.5, 2) -- (4,0) --(6.5,2)--cycle;

\begin{scriptsize}
\draw [fill=black] (0,0) circle (1.25pt);
\draw[color=black] (0,0.3) node {$A$};
\draw [fill=black] (8,0) circle (1.25pt);
\draw[color=black] (8,0.4) node {$B$};
\draw [fill=black] (5,4) circle (1.25pt);
\draw[color=black] (5,4.3) node {$C$};
\draw[color=black] (4,-0.2) node {$c$};
\draw[color=black] (2.9,2.7) node {$b$};
\draw[color=black] (6.3,2.7) node {$a$};
\draw [fill=black] (4,0) circle (1.25pt);
\draw[color=black] (4.05,0.35) node {$m_{C}$};
\draw [fill=black] (6.5,2) circle (1.25pt);
\draw[color=black] (6.7,2.3) node {$m_{A}$};
\draw [fill=black] (2.5,2) circle (1.25pt);
\draw[color=black] (2.45,2.3) node {$m_{B}$};
\draw [fill=black] (6,0) circle (1.25pt);
\draw[color=black] (6.05,0.3) node {$Z$};
\draw [fill=black] (7.25,1) circle (1.25pt);
\draw[color=black] (7.35,1.25) node {$X$};
\draw[color=black] (6.3,1.5) node {$T_1$};
\draw [fill=black] (5.427,1.14) circle (1.5pt);
\draw[color=black] (5.8,1) node {$T_2$};
\draw [fill=black] (6.26,1.81) circle (1.5pt);
\draw [fill=black] (4.28,3.43) circle (1.25pt);
\draw[color=black] (4.23,3.772) node {$Y$};
\end{scriptsize}
\end{tikzpicture}
 \caption{Proof of \Cref{clm2}.}
 \label{f4}
\end{center}
\end{figure}

\noindent 
Let $T_1 = m_Am_C \cap XY$ and $T_2 = m_Am_C \cap YZ$. 
Then $\ar{XT_1T_2Z}\le \ar{Bm_Am_C}$ and by $|T_1T_2|\le|m_Am_C|$ we obtain that
$\ar{T_2 T_1Y}\le \ar{m_Am_Bm_C}$. Thus we have 
$$
    \ar{XYZ}  \leq \ar{Bm_Am_C} + \ar{m_A m_B m_B} = \frac{\ar{ABC}}{2}.
$$
On the other hand, since $c\leq a+b \leq 2b$, the special embedded triangle $AB'C$ satisfies
$$
\ar{AB'C}=\frac{b^2\sin(\sphericalangle CAB)}{2}> \frac{bc\sin(\sphericalangle CAB)}{4} = \frac{\ar{ABC}}{2}.
$$ 
Hence,  $\ar{XYZ}< \ar{AB'C}$, which contradicts the maximality of the area of $XYZ$.
\end{proof}
Lemmas \ref{lemma:Aron} and \ref{clm2} imply that a maximum area embedded isosceles triangle of $ABC$ is either special or its vertex arrangement corresponds to one of the $9$ cases illustrated in \Cref{f5}. 
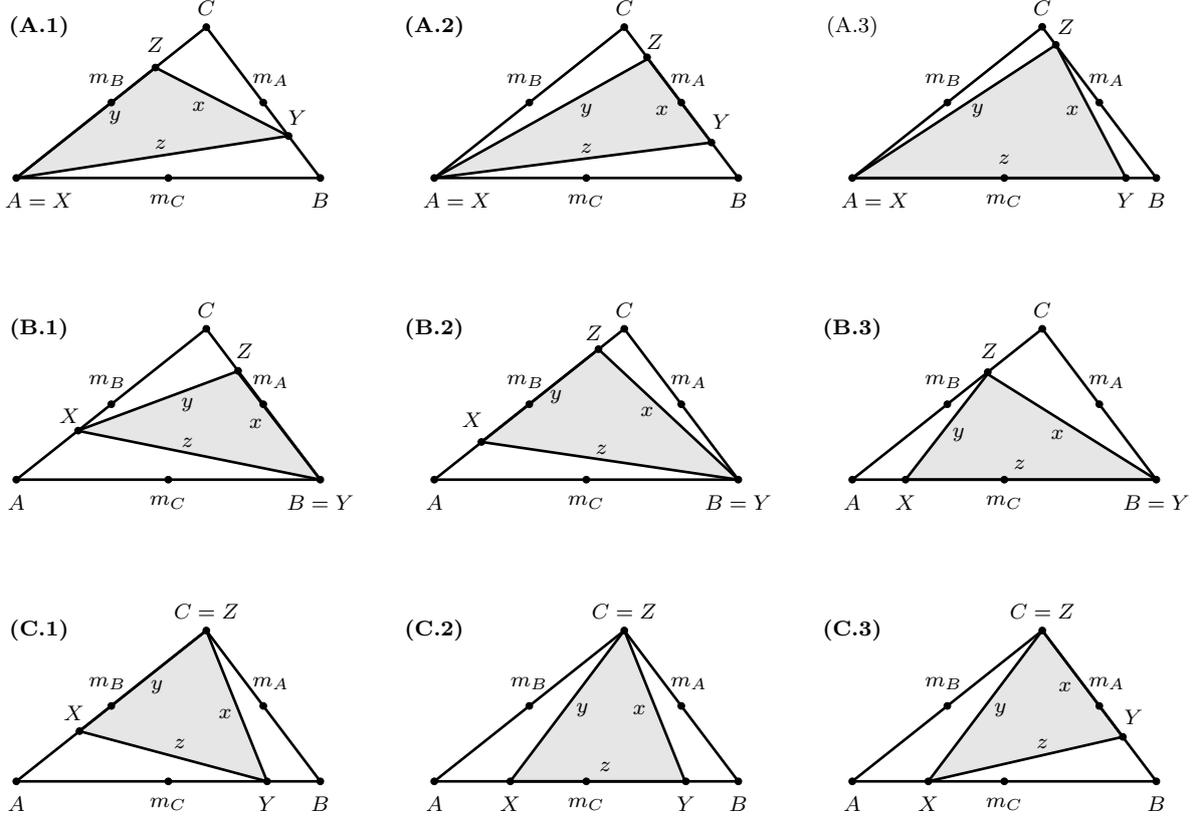
\begin{figure}[H]
\begin{center}
\begin{tikzpicture}
\clip(-0.2,-0.4) rectangle (15.5,10.5);
\fill[line width=1pt,fill=black,fill opacity=0.1] (0,8) -- (3.58,8.56) -- (1.83,9.46) -- cycle;
\fill[line width=1pt,fill=black,fill opacity=0.1] (5.5,8) -- (8.31,9.58) -- (9.15,8.47) -- cycle;
\fill[line width=1pt,fill=black,fill opacity=0.1] (11,8) -- (13.68,9.765) -- (14.6,8) -- cycle;
\fill[line width=1pt,fill=black,fill opacity=0.1] (4,4) -- (0.814,4.65) -- (2.92,5.44) -- cycle;
\fill[line width=1pt,fill=black,fill opacity=0.1] (9.5,4) -- (6.12,4.5) -- (7.66,5.73) -- cycle;
\fill[line width=1pt,fill=black,fill opacity=0.1] (15,4) -- (12.78,5.4256) -- (11.7,4) -- cycle;
\fill[line width=1pt,fill=black,fill opacity=0.1] (2.5,2) -- (3.3,0) -- (0.83,0.666) -- cycle;
\fill[line width=1pt,fill=black,fill opacity=0.1] (8,2) -- (6.5,0) -- (8.8,0) -- cycle;
\fill[line width=1pt,fill=black,fill opacity=0.1] (13.5,2) -- (12,0) -- (14.559,0.59) -- cycle;
\draw [line width=1pt] (0,0)-- (4,0);
\draw [line width=1pt] (2.5,2)-- (0,0);
\draw [line width=1pt] (2.5,2)-- (4,0);
\draw [line width=1pt] (5.5,0)-- (9.5,0);
\draw [line width=1pt] (8,2)-- (5.5,0);
\draw [line width=1pt] (8,2)-- (9.5,0);
\draw [line width=1pt] (11,0)-- (15,0);
\draw [line width=1pt] (13.5,2)-- (11,0);
\draw [line width=1pt] (0,4)-- (4,4);
\draw [line width=1pt] (2.5,6)-- (4,4);
\draw [line width=1pt] (0,4)-- (2.5,6);
\draw [line width=1pt] (5.5,4)-- (9.5,4);
\draw [line width=1pt] (8,6)-- (5.5,4);
\draw [line width=1pt] (9.5,4)-- (8,6);
\draw [line width=1pt] (11,4)-- (15,4);
\draw [line width=1pt] (11,4)-- (13.5,6);
\draw [line width=1pt] (15,4)-- (13.5,6);
\draw [line width=1pt] (0,8)-- (4,8);
\draw [line width=1pt] (2.5,10)-- (4,8);
\draw [line width=1pt] (0,8)-- (2.5,10);
\draw [line width=1pt] (5.5,8)-- (9.5,8);
\draw [line width=1pt] (8,10)-- (9.5,8);
\draw [line width=1pt] (5.5,8)-- (8,10);
\draw [line width=1pt] (11,8)-- (15,8);
\draw [line width=1pt] (15,8)-- (13.5,10);
\draw [line width=1pt] (11,8)-- (13.5,10);
\draw [line width=1pt] (13.5,2)-- (15,0);
\draw [line width=1pt] (0,8)-- (3.58,8.557);
\draw [line width=1pt] (3.58,8.557)-- (1.83,9.464);
\draw [line width=1pt] (1.83,9.464)-- (0,8);
\draw [line width=1pt] (5.5,8)-- (8.31,9.58);
\draw [line width=1pt] (8.31,9.58)-- (9.146,8.47);
\draw [line width=1pt] (9.146,8.47)-- (5.5,8);
\draw [line width=1pt] (11,8)-- (13.68,9.765);
\draw [line width=1pt] (13.678,9.765)-- (14.6,8);
\draw [line width=1pt] (14.6,8)-- (11,8);
\draw [line width=1pt] (4,4)-- (0.8,4.65);
\draw [line width=1pt] (0.81,4.65)-- (2.918,5.44);
\draw [line width=1pt] (2.9,5.44)-- (4,4);
\draw [line width=1pt] (9.5,4)-- (6.12,4.5);
\draw [line width=1pt] (6.12,4.5)-- (7.66,5.73);
\draw [line width=1pt] (7.66,5.73)-- (9.5,4);
\draw [line width=1pt] (15,4)-- (12.78,5.4);
\draw [line width=1pt] (12.78,5.4)-- (11.7,4);
\draw [line width=1pt] (11.7,4)-- (15,4);
\draw [line width=1pt] (2.5,2)-- (3.3,0);
\draw [line width=1pt] (3.3,0)-- (0.83,0.666);
\draw [line width=1pt] (0.83,0.666)-- (2.5,2);
\draw [line width=1pt] (8,2)-- (6.5,0);
\draw [line width=1pt] (6.5,0)-- (8.8,0);
\draw [line width=1pt] (8.8,0)-- (8,2);
\draw [line width=1pt] (13.5,2)-- (12,0);
\draw [line width=1pt] (12,0)-- (14.559,0.59);
\draw [line width=1pt] (14.559,0.59)-- (13.5,2);
\begin{scriptsize}
\draw [fill=black] (0,0) circle (1.25pt);
\draw[color=black] (0.3,2) node {\textbf{(C.1)}};
\draw[color=black] (0,-0.3) node {$A$};
\draw [fill=black] (4,0) circle (1.25pt);
\draw[color=black] (4,-0.3) node {$B$};
\draw [fill=black] (2.5,2) circle (1.25pt);
\draw[color=black] (2.5,2.25) node {$C=Z$};
\draw [fill=black] (5.5,0) circle (1.25pt);
\draw[color=black] (5.5,2) node {\textbf{(C.2)}};
\draw[color=black] (5.5,-0.3) node {$A$};
\draw [fill=black] (9.5,0) circle (1.25pt);
\draw[color=black] (9.5,-0.3) node {$B$};
\draw [fill=black] (8,2) circle (1.25pt);
\draw[color=black] (8,2.25) node {$C=Z$};
\draw [fill=black] (11,0) circle (1.25pt);
\draw[color=black] (11,2) node {\textbf{(C.3)}};
\draw[color=black] (11,-0.3) node {$A$};
\draw [fill=black] (15,0) circle (1.25pt);
\draw[color=black] (15,-0.3) node {$B$};
\draw [fill=black] (13.5,2) circle (1.25pt);
\draw[color=black] (13.5,2.25) node {$C=Z$};
\draw [fill=black] (0,4) circle (1.25pt);
\draw[color=black] (0.3,6) node {\textbf{(B.1)}};
\draw[color=black] (0,3.7) node {$A$};
\draw [fill=black] (4,4) circle (1.25pt);
\draw[color=black] (4,3.7) node {$B=Y$};
\draw [fill=black] (2.5,6) circle (1.25pt);
\draw[color=black] (2.5,6.25) node {$C$};
\draw [fill=black] (5.5,4) circle (1.25pt);
\draw[color=black] (5.5,6) node {\textbf{(B.2)}};
\draw[color=black] (5.5,3.7) node {$A$};
\draw [fill=black] (9.5,4) circle (1.25pt);
\draw[color=black] (9.5,3.7) node {$B=Y$};
\draw [fill=black] (8,6) circle (1.25pt);
\draw[color=black] (8,6.25) node {$C$};;
\draw [fill=black] (11,4) circle (1.25pt);
\draw[color=black] (11,6) node {\textbf{(B.3)}};
\draw[color=black] (11,3.7) node {$A$};
\draw [fill=black] (15,4) circle (1.25pt);
\draw[color=black] (15,3.7) node {$B=Y$};
\draw [fill=black] (13.5,6) circle (1.25pt);
\draw[color=black] (13.5,6.25) node {$C$};
\draw [fill=black] (0,8) circle (1.25pt);
\draw[color=black] (0.3,10) node {\textbf{(A.1)}};
\draw[color=black] (0.3,7.7) node {$A=X$};
\draw [fill=black] (4,8) circle (1.25pt);
\draw[color=black] (4,7.7) node {$B$};
\draw [fill=black] (2.5,10) circle (1.25pt);
\draw[color=black] (2.5,10.25) node {$C$};
\draw [fill=black] (5.5,8) circle (1.25pt);
\draw[color=black] (5.5,10) node {\textbf{(A.2)}};
\draw[color=black] (5.8,7.7) node {$A=X$};
\draw [fill=black] (9.5,8) circle (1.25pt);
\draw[color=black] (9.5,7.7) node {$B$};
\draw [fill=black] (8,10) circle (1.25pt);
\draw[color=black] (8,10.25) node {$C$};
\draw [fill=black] (11,8) circle (1.25pt);
\draw[color=black] (11,10) node {(A.3)};
\draw[color=black] (11.3,7.7) node {$A=X$};
\draw [fill=black] (15,8) circle (1.25pt);
\draw[color=black] (15,7.7) node {$B$};
\draw [fill=black] (13.49899,10.00426) circle (1.25pt);
\draw[color=black] (13.5,10.25) node {$C$};
\draw [fill=black] (2,0) circle (1.25pt);
\draw[color=black] (2,-0.3) node {$m_C$};
\draw [fill=black] (7.5,0) circle (1.25pt);
\draw[color=black] (7.5,-0.3) node {$m_C$};
\draw [fill=black] (13,0) circle (1.25pt);
\draw[color=black] (13,-0.3) node {$m_C$};
\draw [fill=black] (2,4) circle (1.25pt);
\draw[color=black] (2,3.7) node {$m_C$};
\draw [fill=black] (2,8) circle (1.25pt);
\draw[color=black] (2,7.7) node {$m_C$};
\draw [fill=black] (7.5,8) circle (1.25pt);
\draw[color=black] (7.5,7.7) node {$m_C$};
\draw [fill=black] (7.5,4) circle (1.25pt);
\draw[color=black] (7.5,3.7) node {$m_C$};
\draw [fill=black] (13,4) circle (1.25pt);
\draw[color=black] (13,3.7) node {$m_C$};
\draw [fill=black] (13,8) circle (1.25pt);
\draw[color=black] (13,7.7) node {$m_C$};
\draw [fill=black] (3.25,1) circle (1.25pt);
\draw[color=black] (3.35,1.3) node {$m_A$};
\draw [fill=black] (1.25,1) circle (1.25pt);
\draw[color=black] (1.2,1.3) node {$m_B$};
\draw [fill=black] (3.25,5) circle (1.25pt);
\draw[color=black] (3.35,5.3) node {$m_A$};
\draw [fill=black] (1.25,5) circle (1.25pt);
\draw[color=black] (1.2,5.3) node {$m_B$};
\draw [fill=black] (6.75,1) circle (1.25pt);
\draw[color=black] (6.75,1.3) node {$m_B$};
\draw [fill=black] (8.75,1) circle (1.25pt);
\draw[color=black] (8.85,1.3) node {$m_A$};
\draw [fill=black] (8.75,5) circle (1.25pt);
\draw[color=black] (8.85,5.3) node {$m_A$};
\draw [fill=black] (6.75,5) circle (1.25pt);
\draw[color=black] (6.7,5.3) node {$m_B$};
\draw [fill=black] (14.25,1) circle (1.25pt);
\draw[color=black] (14.35,1.3) node {$m_A$};
\draw [fill=black] (12.25,1) circle (1.25pt);
\draw[color=black] (12.2,1.3) node {$m_B$};
\draw [fill=black] (14.25,5) circle (1.25pt);
\draw[color=black] (14.35,5.3) node {$m_A$};
\draw [fill=black] (12.25,5) circle (1.25pt);
\draw[color=black] (12.2,5.3) node {$m_B$};
\draw [fill=black] (14.25,9) circle (1.25pt);
\draw[color=black] (14.35,9.3) node {$m_A$};
\draw [fill=black] (12.25,9) circle (1.25pt);
\draw[color=black] (12.2,9.3) node {$m_B$};
\draw [fill=black] (8.75,9) circle (1.25pt);
\draw[color=black] (8.85,9.3) node {$m_A$};
\draw [fill=black] (6.75,9) circle (1.25pt);
\draw[color=black] (6.7,9.3) node {$m_B$};
\draw [fill=black] (3.25,9) circle (1.25pt);
\draw[color=black] (3.35,9.3) node {$m_A$};
\draw [fill=black] (1.25,9) circle (1.25pt);
\draw[color=black] (1.2,9.3) node {$m_B$};
\draw [fill=black] (3.582,8.557) circle (1.25pt);
\draw[color=black] (3.7,8.8) node {$Y$};
\draw [fill=black] (1.83,9.464) circle (1.25pt);
\draw[color=black] (1.83,9.76) node {$Z$};
\draw[color=black] (1.9,8.45) node {$z$};
\draw[color=black] (2.4,8.95) node {$x$};
\draw[color=black] (1.3,8.8) node {$y$};
\draw [fill=black] (8.3,9.6) circle (1.25pt);
\draw[color=black] (8.4,9.8) node {$Z$};
\draw [fill=black] (9.146,8.47) circle (1.25pt);
\draw[color=black] (9.28,8.75) node {$Y$};
\draw[color=black] (7.5,8.9) node {$y$};
\draw[color=black] (8.5,8.9) node {$x$};
\draw[color=black] (7.5,8.4) node {$z$};
\draw [fill=black] (13.678,9.765) circle (1.25pt);
\draw[color=black] (13.81,10) node {$Z$};
\draw [fill=black] (14.6,8) circle (1.25pt);
\draw[color=black] (14.6,7.7) node {$Y$};
\draw[color=black] (12.65,8.9) node {$y$};
\draw[color=black] (13.9,8.9) node {$x$};
\draw[color=black] (13,8.25) node {$z$};
\draw [fill=black] (0.814,4.65) circle (1.25pt);
\draw[color=black] (0.7,4.85) node {$X$};
\draw [fill=black] (2.918,5.443) circle (1.25pt);
\draw[color=black] (3,5.7) node {$Z$};
\draw[color=black] (2.25,4.5) node {$z$};
\draw[color=black] (2.25,5) node {$y$};
\draw[color=black] (3.15,4.75) node {$x$};
\draw [fill=black] (6.12,4.5) circle (1.25pt);
\draw[color=black] (6,4.8) node {$X$};
\draw [fill=black] (7.66,5.73) circle (1.25pt);
\draw[color=black] (7.6,5.95) node {$Z$};
\draw[color=black] (7.7,4.4) node {$z$};
\draw[color=black] (7.1,5.1) node {$y$};
\draw[color=black] (8.3,4.9) node {$x$};
\draw [fill=black] (12.782,5.4256) circle (1.25pt);
\draw[color=black] (12.8,5.7) node {$Z$};
\draw [fill=black] (11.7,4) circle (1.25pt);
\draw[color=black] (11.7,3.7) node {$X$};
\draw[color=black] (13.7,4.6) node {$x$};
\draw[color=black] (12.4,4.6) node {$y$};
\draw[color=black] (13.2,4.2) node {$z$};
\draw [fill=black] (0.832,0.666) circle (1.25pt);
\draw[color=black] (0.76,0.9) node {$X$};
\draw [fill=black] (3.3,0) circle (1.25pt);
\draw[color=black] (3.3,-0.3) node {$Y$};
\draw[color=black] (2.75,0.9) node {$x$};
\draw[color=black] (2.15,0.5) node {$z$};
\draw[color=black] (1.85,1.25) node {$y$};
\draw [fill=black] (6.5,0) circle (1.25pt);
\draw[color=black] (6.5,-0.3) node {$X$};
\draw [fill=black] (8.81,0) circle (1.25pt);
\draw[color=black] (8.81,-0.3) node {$Y$};
\draw[color=black] (7.45,0.95) node {$y$};
\draw[color=black] (7.75,0.2) node {$z$};
\draw[color=black] (8.2,0.95) node {$x$};
\draw [fill=black] (12,0) circle (1.25pt);
\draw[color=black] (12,-0.3) node {$X$};
\draw [fill=black] (14.559,0.59) circle (1.25pt);
\draw[color=black] (14.7,0.85) node {$Y$};
\draw[color=black] (12.95,0.95) node {$y$};
\draw[color=black] (13.5,0.5) node {$z$};
\draw[color=black] (13.8,1.25) node {$x$};
\end{scriptsize}
\end{tikzpicture}
\end{center}
 \caption{The $9$ possible arrangements of the points $X, Y, Z$ in a given triangle $ABC$.}
 \label{f5}
\end{figure}
\noindent
To complete the proof of \Cref{tfo}(ii), it remains to prove that none of the arrangements depicted on  \Cref{f5} can be optimal. 
We prove this for each of the $9$ cases, separately. Note that in some instances, we will refer to special embedded triangles using their specific labeling introduced in \Cref{sec:special-embedded}.

\noindent
\textbf{Case A:} \textit{The common vertex of $ABC$ and $XYZ$ is $A=X$. }

\medskip
\noindent
{\textbf{Subcase A.1:}}
\textit{ $Y\in BC$ and $Z\in AC$.}\\
 Observe that since $b<c$, the orthogonal projection of $A$ onto $CB$ is contained in $Cm_A$, which implies that $\sphericalangle AYB$ is obtuse. Thus, we can rotate  $XYZ$ about $X$ such that two of its vertices get to the interior of $ABC$ and so, by \Cref{lemma:Aron}, $XYZ$ cannot be of maximum area.

\medskip
\noindent
\textbf{\textbf{Subcase A.2:}} \textit{Both $Y$ and $Z$ are in $BC$.}\\
If $y = z$, then we can increase $\ar{XYZ}$ by moving $Z$ towards $C$ and $Y$ towards $B$ while maintaining $|XZ| = |XY|$, since $\alpha=\sphericalangle CAB < 90^{\circ}$.  If $ABC$ is acute, then we can do this until the vertices $Z$ and $C$ will coincide, and triangle $XYZ$ will be the same as the special embedded triangle $A\overline{B}C$. If $ABC$ is non-acute, then $y \neq z$.
Clearly, $|AZ|=y>|ZB|>|YZ|=x.$
Hence, $x \neq y$. A similar argument shows that $x \neq z$.

\medskip
\noindent
\textbf{\textbf{Subcase A.3:}}\textit{ $Y\in AB$ and $Z\in BC$.}\\
Since $a<b$, the orthogonal projection $\widehat{Z}$ of $Z$ to the line segment $AB$ lies in $m_CB$.

If $x=y$, then  $|AY| = 2|A\hat Z| > 2 |Am_C| = |AB|$, a contradiction to $Y \in AB$.

If $x=z$, then the altitude with base $z$ in $XYZ$ is smaller than the altitude with base $c$ in $ABC$. On the other hand, $x=z<a$,  if $\sphericalangle ZYB\ge 90^{\circ}$. In this case, the special embedded triangle $A''BC$ satisfies $\ar{A''BC}> \ar{XYZ}$. Otherwise, $x=z<y$ (as $\sphericalangle AYZ>90^{\circ}$) and $y<c$. Let $Y'$ be the point in $AB$ that is defined by the equality $|AY'|=|AZ|$. (The existence of $Y'\in AB$ is a consequence of $y<c$.) Then, $\ar{XY'Z}> \ar{XYZ}$. In both cases it follows that the area of $XYZ$ cannot be optimal.

If $y=z$, consider the special embedded triangle $AB'C$,  define $l$ to be the line parallel to $B'C$ going through $Y$ and let $Z'=l\cap BC$, see \Cref{f6}. 
Since $Z'\in CZ$, we have
$$
\ar{XYZ}<\ar{XYZ'}=\ar{AB'C} \cdot \frac{b+|B'Y|}{b} \cdot \frac{c-b-|B'Y|}{c-b}.
$$
The inequality follows from the fact that $Z'\in CZ$. 
Therefore, the altitude of $XYZ$ with base $z$ is greater than the altitude of $XYZ'$ with base $z$.
Thus, it is enough to show that 
$$\frac{b+|B'Y|}{b}\cdot \frac{c-b-|B'Y|}{c-b}<1.$$ As $b>0$ and $c-b>0$, this is equivalent to $|B'Y|(2b-c+|B'Y|)>0,$
which follows from the triangle inequality $c<a+b<2b$.
\begin{figure}[H]
\begin{center}
 \includegraphics[width=0.4\textwidth]{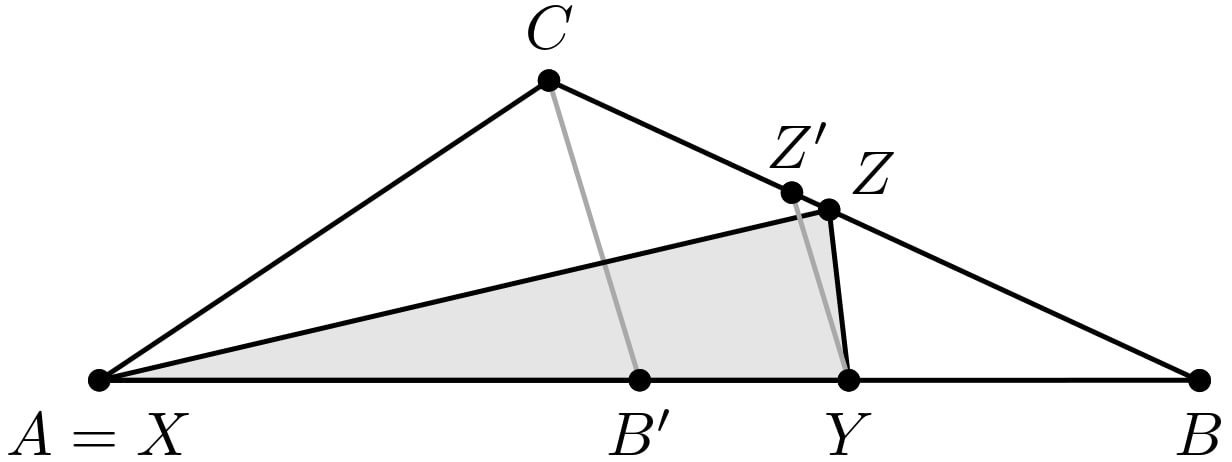}
\end{center}
 \caption{Illustration for Subcase A.3.}
 \label{f6}
\end{figure}

\medskip
\noindent
\textbf{Case B:} \textit{The common vertex of $ABC$ and $XYZ$ is $B=Y$.}

\medskip
\noindent
\textbf{Subcase B.1:} \textit{$X\in AC$ and $Z\in BC$.}\\
Since $a<c$, we have that $\sphericalangle AXY>90^{\circ}$, and hence, we can rotate the triangle $XYZ$ around $Y$ so that the image of the vertices $X,Z$ will be inside of $ABC$. As in Subcase A.1, this implies that the area of $XYZ$ is not optimal.

\medskip
\noindent
\textbf{Subcase B.2:} \textit{Both $X$ and $Z$ are in $AC$.}\\
Observe that  $b<c$ implies that $A$ and $C$ are on the same side of the perpendicular bisector of $BC$. This implies that $|XY|=z>|XC|>|XZ|=y$. 
If $x=z$, we can `open' $\sphericalangle XYZ$ as in  Subcase A.2 and get that 
$\ar{XYZ}<\ar{\overline{\overline{A}}BC}$. 
Hence, we can assume that $x=y$.

If the triangle $ABC$ is non-acute, then  consider the special embedded triangle $ABC_1$. Since the altitudes of $ABC_1$ and $XYZ$ from vertex $B=Y$ are equal, and $x=y<|BC_1|=|AC_1|$ (as
$\sphericalangle BCA\ge 90^\circ$), 
 we have that $\ar{XYZ}<\ar{ABC_1}$. 

If $ABC$ is acute, let $\hat{B}$ denote the orthogonal projection of $B$ onto $AC$. If $Z\in A\hat{B}$, then we can slightly rotate $XYZ$  around $Y$ (as $\sphericalangle YXA> \sphericalangle YZA>90^{\circ}$). Thus, by \Cref{lemma:Aron}(iv), the area of $XYZ$ is not maximal. 
Thus, we can assume that $Z\in C\hat{B}$, that is, $\sphericalangle YZA\le 90^\circ$. Similarly as above, this implies that $x=|YZ|<a=|BC|$ and thus the special embedded triangle $A'BC$ satisfies $\ar{XYZ}<\ar{A'BC}$. 

\medskip
\noindent\textbf{Subcase B.3:} \textit{$X\in AB$ and $Z\in AC$.}\\
If $y=z$, then, since $\sphericalangle CAB<\min(\sphericalangle AXZ, \sphericalangle ZXY)$, we get that $y=|XZ| < |AZ| < b=|AC|$, which immediately implies that the special embedded triangle $AB'C$ satisfies $\ar{XYZ}< \ar{AB'C}$.

Now we assume that $x=z$. If $A$ and $Z$ lie on the same side of the perpendicular bisector of $AB$, then we can reflect $XYZ$ to this perpendicular bisector. We denote this reflection by $X'Y'Z'$. Clearly, $X',Y' \in AB$, and $Z'$ is inside of $ABC$, which implies that $\ar{XYZ}$ is not maximal. 
If $Z$ is on the perpendicular bisector of $AB$, then $XYZ$ is strictly contained in the special embedded triangle $ABC_1$, so $\ar{XYZ} < \ar{ABC_1}$.
If $Z$ and $C$ are on the same side of the perpendicular bisector of $AB$, then $x=z<|AZ|<|AC|=b$, and hence 
$\ar{XYZ}<\ar{AB'C}$.

It remains to handle the case $x=y$. We show that $\ar{XYZ}<\ar{ABC_1}$. The condition $x=y$ implies that $Z\in C_1C$. 
Plainly, $z=c-|AX|$. 
Denote the lengths of the altitudes from $C_1$ in $ABC_1$ and from $Z$ in $XYZ$ by $h_{C_1}$ and $h_Z$, respectively. Clearly, we get $h_Z=h_{C_1}\frac{c+|AX|}{c}$, and hence  $$\ar{XYZ}=\ar{ABC_1}\frac{c+|AX|}{c}\cdot \frac{c-|AX|}{c}=\ar{ABC_1}\frac{c^2-|AX|^2}{c^2} <\ar{ABC_1}.$$

\medskip
\noindent
\textbf{Case C:} \textit{The common vertex of $ABC$ and $XYZ$ is $C=Z$. }

\medskip
\noindent
\textbf{Subcase C.1:} \textit{$X\in AC$ and $Y\in AB$.}\\
If $Y$ and $B$ are on the same side of the altitude from $C$, then we can rotate $XYZ$ about $Z$ so that $X$ and $Y$ get to the interior of $ABC$ which by \Cref{lemma:Aron}(iv) implies that $XYZ$ is not optimal. If $Y$ and $B$ are on different sides of the altitude from $C$, then $XYZ$ is strictly contained in the special embedded triangle $AB'C$.

\medskip
\noindent
\textbf{Subcase C.2:} \textit{Both $X$ and $Y$ are contained in $AB$.}\\
If $x=y$, then we can open $\sphericalangle YZX$, which increases the area of $XYZ$, so $\ar{XYZ}$ is not maximal.
Suppose that $x=z$. If $Y$ and $A$ are on the same side of the altitude from $C$, then $XYZ$ is strictly contained in the special embedded triangle $AB'C$. If $Y$ and $A$ lie on different sides of the altitude, then the special embedded triangle $A''BC$ satisfies $\ar{XYZ}<\ar{A''BC}$. Indeed, their altitudes from $C$ are the same, and for their bases we have $x=z<a$. Thus $XYZ$ is not maximal. 
Analogously, for $y=z$ a similar argument shows that 
$\ar{XYZ}<\ar{AB'C}$.

\medskip
\noindent
\textbf{Subcase C.3:}\textit{ $X\in AB$ and $Y\in BC$.}\\
We can rotate $XYZ$ about $Z$ such that the images of $X$ and $Y$ lie in the interior of $ABC$, and so, by \Cref{lemma:Aron}(iv), we get that $\ar{XYZ}$ is not maximal.

\medskip
\noindent
We have shown that none of the triangles $XYZ$ of the $9$ cases in \Cref{f5} is a maximum area embedded isosceles triangle of $ABC$, which completes the proof of \Cref{tfo}(ii). \qed

\section{Maximum perimeter embedded isosceles triangles  \\ -- Proof of \Cref{tfo}(iii)}\label{sec:max-perim}
In this section, we prove that for any triangle $ABC$, any maximum perimeter isosceles triangle $XYZ$ embedded in $ABC$ shares a vertex and the angle at that vertex with $ABC$. 
 First we collect the observations in Lemmas \ref{lem:01} and \ref{lem:m1} concerning maximum perimeter embedded isosceles triangles. 
\begin{lem}\label{lemmas}
Let $XYZ$ be a maximum perimeter isosceles triangle embedded in $ABC$.  Then~\vspace{-.5em}
\begin{enumerate}[(i)]
 \setlength{\itemsep}{1pt}
  \setlength{\parskip}{1pt}
    \item  each side of $ABC$ contains a vertex of $XYZ$;
    \item\label{L2} no vertex of the triangle $XYZ$ lies in the interior of the triangle $ABC$;   
    \item \label{L4} there is a side of $ABC$ which contains a side of $XYZ$; 
    \item $ABC$ and $XYZ$ share a vertex.
\end{enumerate}
\end{lem}
We will show that an  isosceles triangle embedded in $ABC$ which does not share an angle with $ABC$ cannot be of minimum perimeter.
Notice that if $ABC$ and $XYZ$ share at least two vertices, then, by \Cref{lemmas}(\ref{L2}), they also share an angle, so we are done. Thus, it is enough to consider those cases where the triangles $XYZ$ and $ABC$ share exactly one vertex, without loss of generality the common vertex is $A$. Note that in this section, we do not assume a special labeling of $ABC$, in particular, we do not necessarily have $|BC|<|AC|<|AB|$.
On the other hand, we assume that $XYZ$ is labeled so that  $|XY| = |YZ|$. We consider the following cases, separately:\\

\noindent
\textbf{Case A:} {\it $X$ and $Z$ lie on the same side of $ABC$.}\\  
 We can always rotate $X$ or $Z$ (for simplicity, assume it is $X$) about $Y$ so that the rotated point $X'$ lies in the interior of $ABC$ and   $\sphericalangle XYZ < \sphericalangle X'YZ$, see \Cref{fig:A}. By the Hinge theorem\footnote{Hinge theorem: Let $XYZ$ be a triangle and let $X'Y'Z'$ be another triangle such that $XY=X'Y', YZ=Y'Z'$, and $\sa XYZ <\sa X'Y'Z'$. Then $\per{XYZ}<\per{X'Y'Z'}$. }, 
 $\per{XYZ}<\per{X'YZ}$.
\begin{figure}[h!]
    \centering
    \includegraphics[width=0.4\textwidth]{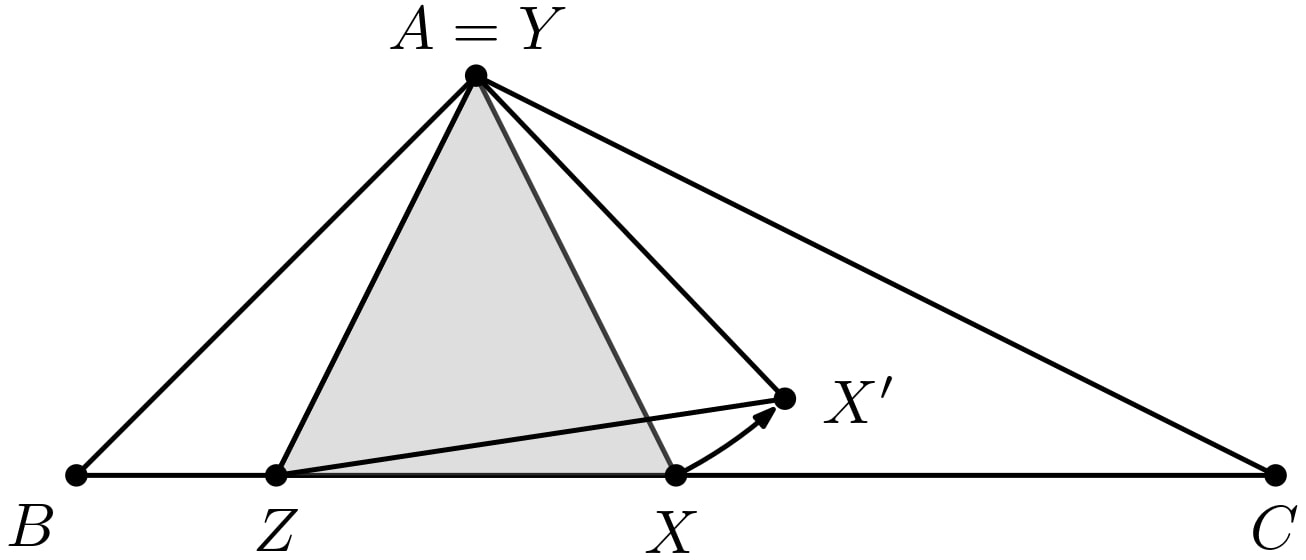}\hspace{2em} \includegraphics[width=0.4\textwidth]{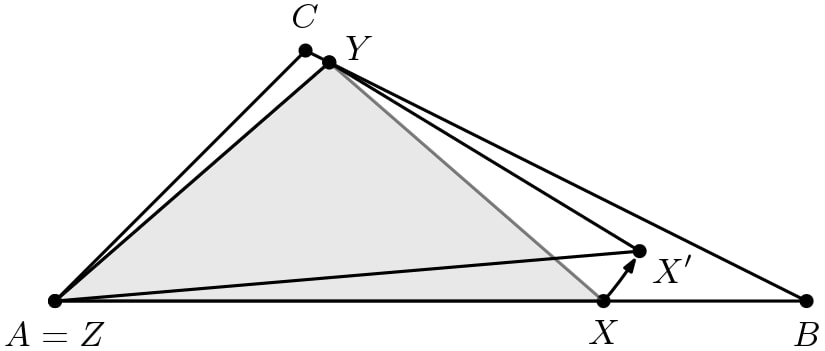}
    \caption{Illustration for Case A}
    \label{fig:A}
\end{figure}

\medskip
\noindent
\textbf{Case B:}  \textit{$X$ and $Z$ lie  on different sides of $ABC$.}\\
We will make use of the following classical lemma on the perimeter of the Minkowski sum of convex bodies.
    \begin{lem}[see e.g.{\cite[exercise 4-7]{YB61}}]\label{lemma:Minkowski}
        Let $K_1$ and $K_2$ be two convex bodies in the plane and let $K=\frac{K_1+K_2}{2}$ be the Minkowski mean of $K_1$ and $K_2$. Then the perimeter of $K$ is equal to the arithmetic mean of the perimeters of $K_1$ and $K_2$. If $K_1$ and $K_2$ are not homothetic triangles, then $K$ is a convex polygon with at least four sides.
    \end{lem}
\medskip
\noindent
\textbf{Subcase B.1:} {\it The common vertex of $ABC$ and $XYZ$ is $A=Y$.}

\noindent
If none of $X$ and $Z$ is on the side opposite to $Y$, then $XYZ$ and $ABC$ have a common angle at $Y$.
Thus, we can assume that either $X$ or $Z$ is on the side opposite to $Y$, say it is $X$.

The idea is to show that the triangle $XYZ$ is strictly contained in the Minkowski mean $M$ of two other non-homothetic isosceles triangles embedded in $ABC$, thus, by \Cref{lemma:Minkowski}, one of these two must have a strictly larger perimeter (by the fact that if $\mathcal C_1,\mathcal C_2$ are two convex planar sets such that $\mathcal C_1\subseteq \mathcal C_2$, then $\per{\mathcal C_1}\leq\per{\mathcal C_2}$ 
     \cite[12.10.2]{berger09geometryII}).
     
 Let $\delta$ be a constant satisfying $\delta < \min\{ |XB|, |XC|\}$. Define the points $X^1$ and $X^2$ by translating $X$ by $\delta$ towards $C$ and $B$, respectively. Let $Z^1$ and $Z^2$ be such that they are contained on the side $AB$ with $|YZ^1| = |YX^1|$ and $|Y Z^2| = |Y X^2|$, see \Cref{fig:A2}.
 \begin{figure}[h!]
     \centering
      \includegraphics[width=0.4\textwidth]{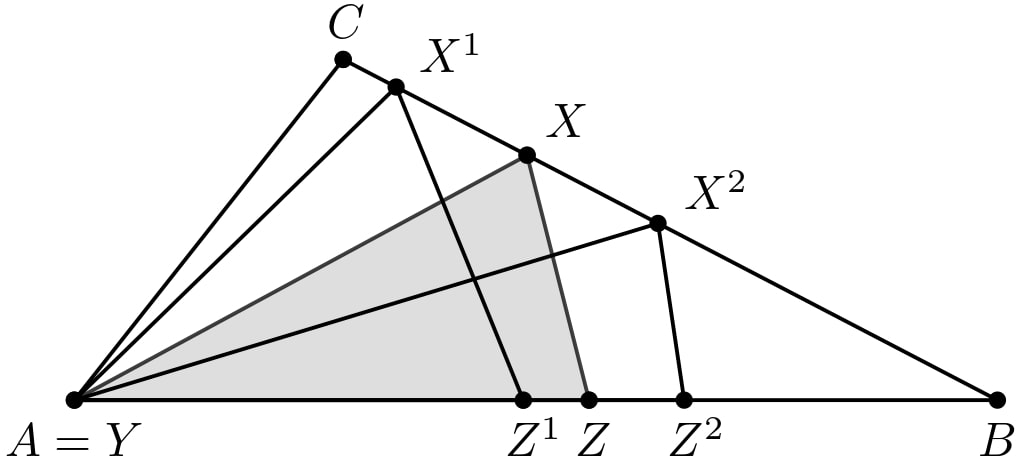}
     \caption{Illustration for Subcase B.1.}
     \label{fig:A2}
 \end{figure}
         Let $M$ be the Minkowski mean of $ X^1YZ^1$ and $X^2YZ^2$.  The vertex $Y$ is contained in both triangles, thus it is also contained in $M$. It is also easy to see that $X \in M$ since $ X = \frac 1 2 (X^1 + X^2)$.
         We show that $Z$ is contained in the segment between $Y$ and $\frac 1 2 (Z^1 + Z^2)$, which implies $Z \in M$.
         To this end, observe that the segment $YX$ is a median of the triangle $X^1YX^2$ and thus $|YX| < \frac 1 2 \round{|YX^1| + |YX^2|}$, which directly implies that $|YZ| < \frac 1 2 \round{|YZ^1| + |YZ^2|}$.

\medskip
\noindent
\textbf{Subcase B.2:} {\it The common vertex of $ABC$ and $XYZ$ is $A=Z$ and both $X$ and $Y$ are in the interior of the side of $ABC$ opposite to $Z$.}\\
        
\noindent        
Define the points $X^1$ and $X^2$ by translating $X$ by $\delta$ towards $C$ and $B$, respectively. We choose $\delta$ to be small enough such that there are points  $Y^1,Y^2$ in the segment $BC$ with  $|Y^1Z| = |Y^1 X^1|$ and $|Y^2 Z| = |Y^2 X^2|$, see \Cref{fig:B2}. Let $M$ be the Minkowski mean of $ X^1YZ^1$ and $X^2YZ^2$. As before, it is clear that the vertices $X$ and $Z$ are contained in $M$.
         
         \begin{figure}[ht]
        \centering
         \includegraphics[width=0.4\textwidth]{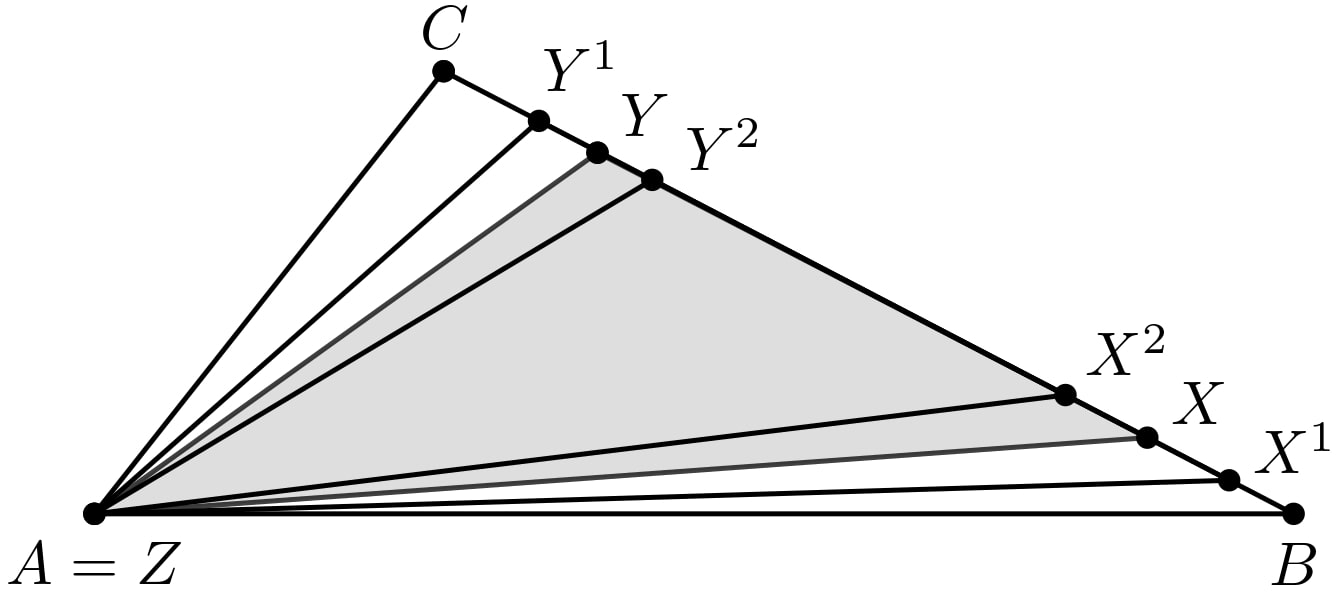}
         \caption{Illustration for Subcase B.2.}\label{fig:B2}
         \end{figure}
         
         \noindent
         To argue that $Y \in M$, we shall show that $Y$ is contained in the segment between $X$ and $\frac 1 2 (Y^1 + Y^2)$. To simplify the calculations, we move and scale the triangle so that $A = (0,1)$, $B=(b,0)$, $C = (c,0)$ and $X = (x',0)$ with $b < x' < c$. Note that since $\sphericalangle ZXY$ is acute, $x' < 0$. For each $b < x < 0$, let $X_x = (x,0)$ and $Y_x$ be the point in $BC$ such that $|ZY_x| = |Y_x X_x|$ and define $f(x) = |X_x Y_x|$, see \Cref{fig:B2-axes}. Observe that $Y$ is contained in the segment between $X$ and $\frac 1 2 (Y^1 + Y^2)$ iff $\frac 1 2 (f(x' - \delta) + f(x'+\delta)) > f(x')$. Thus, it is sufficient to show that $f(x)$ is a convex function on $(b,0)$.   
         \begin{figure}
             \centering
              \includegraphics[width=0.45\textwidth]{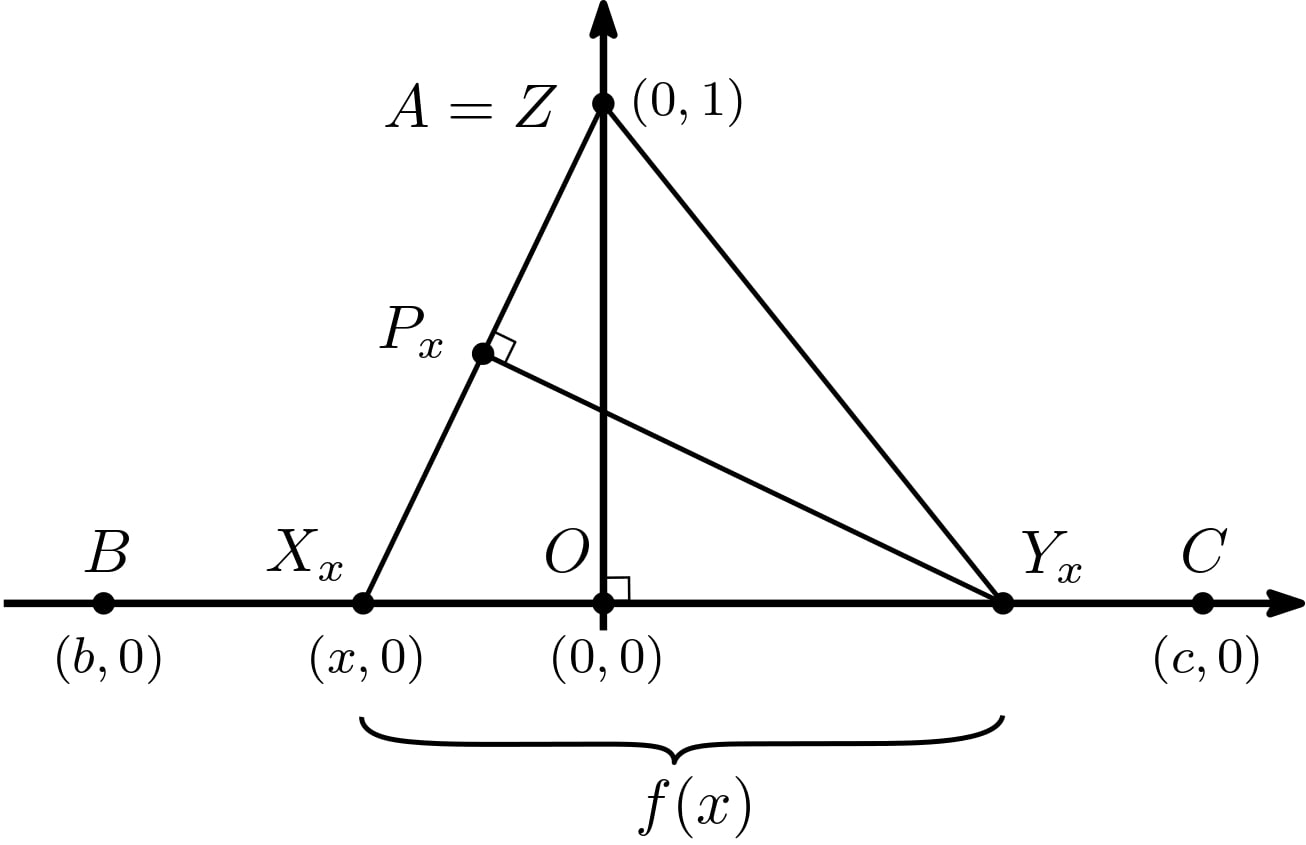}
             \caption{Embedding the instance in $\RR^2$.}
             \label{fig:B2-axes}
         \end{figure}
         To find an analytic formula for $f(x)$, we introduce some auxiliary points. Let $O = (0,0)$ and  $P_x$ be the orthogonal projection of $Y_x$ to the segment $X_xZ$. Note that $P_x$ is the midpoint of $XZ$. Then the triangles $X_xP_xY_x$ and $X_x O Z$ are similar, which yields
         $$
            f(x) = |Y_x X_x| = |X_xZ| \cdot \frac{|X_xP_x|}{|X_x O|} = \sqrt{1+x^2} \cdot \frac{\sqrt{1+x^2}/2}{-x} = \frac{1+x^2}{-2x}
         $$
            The second derivative of $f$ is $f''(x) = - 1/x^3$, thus $f(x)$ is convex on the interval $(b,0)$.
        
\medskip
\noindent
\textbf{Subcase B.3:} {\it The common vertex of $ABC$ and $XYZ$ is $A=Z$ and $X$, $Y$ lie in the interior of different sides of $ABC$.}\\ 
    Firstly, since $X$ and $Z$ lie on different sides of $ABC$, we get that $X$ is on the side opposite to $Z$, see \Cref{fig:B3}. If  $\sphericalangle AXB$ is obtuse, then we can rotate the triangle $XYZ$ about $Z$ and obtain a copy of $XYZ$ which has two vertices in the interior of $ABC$, thus by \Cref{lemmas}, $XYZ$ cannot be of maximum perimeter. Therefore,  $\sphericalangle AXB$ and consequently $\sphericalangle ACB$ are acute.
         
    Define the points $X^1$ and $X^2$ by translating $X$ by $\delta$ towards $C$ and $B$, respectively. We choose an increment $\delta \in (0,1/c)$ which is small enough such that there are points  $Y^1,Y^2$ in the segment $AC$ with  $|Y^1Z| = |Y^1 X^1|$ and $|Y^2 Z| = |Y^2 X^2|$. Let $M$ be the Minkowski mean of $ X^1YZ^1$ and $X^2YZ^2$. The vertices $X$ and $Z$ are clearly contained in $M$.
          \begin{figure}[h!]
     \centering
      \includegraphics[width=0.33\textwidth]{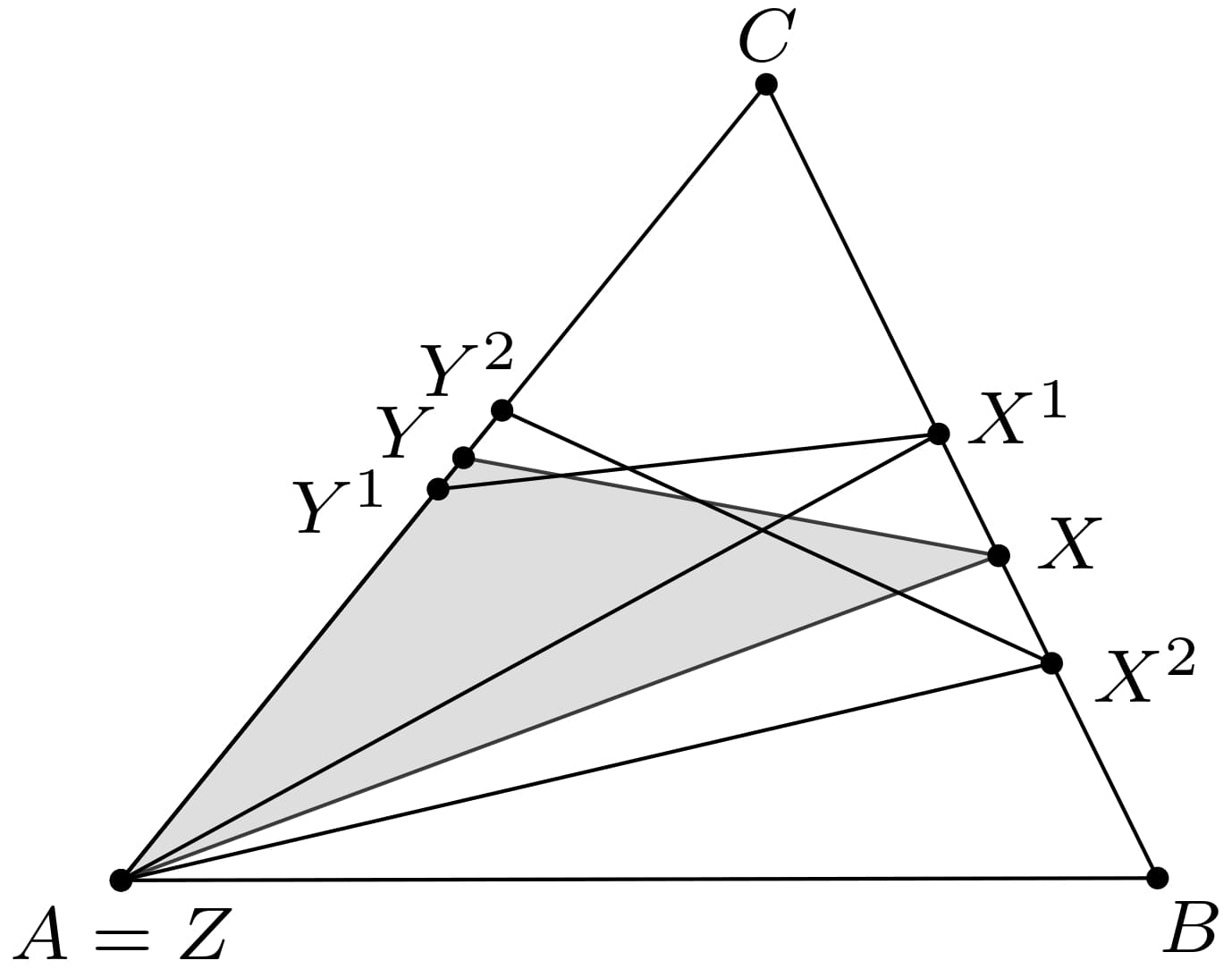}
         \caption{Illustration for Case B.3.}
     \label{fig:B3}
 \end{figure}
 
         To prove that $Y \in M$, we shall show that $Y$ is contained in the segment between $Z$ and $\frac 1 2 (Y^1 + Y^2)$. Again, we translate and scale of the triangle so that $A = (0,1)$, $B=(b,0)$, $C = (c,0)$ and $X = (x',0)$ with $b < x' < c$. Since $\sphericalangle AXB$ is acute, we have $x'\geq 0$. For each $x \in [x' -\delta, x'+\delta]$, let $X_x = (x,0)$ and $Y_x$ be the point in $ZC$ such that $|Z Y_x| = |Y_x X_x|$ and define $f(x) = |Z Y_x|$, see \Cref{fig:B3-axes}. 
        Note that since $x'\geq 0$ and $\delta$ is smaller than $1/c$, each $x \in [x' -\delta, x'+\delta]$ satisfies $-1/c < x$.  
        We want to show that the function $f(x)$ is convex, which then directly implies that $Y$ is contained in the segment between $Z$ and $\frac 1 2 (Y^1 + Y^2)$. 
        \begin{figure}
            \centering
         \includegraphics[width=0.4\textwidth]{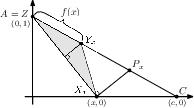}
            \caption{Embedding the instance in $\RR^2$.}
            \label{fig:B3-axes}
        \end{figure}
        Let $P_x  = (p_1(x),p_2(x))$ be a point on $AC$ such that the segment $P_xX_x$ is orthogonal  to $AX_x$. 
        Note that $P_x$ satisfies $|Z P_x| = 2f(x)$.
        
        Let $\gamma$ denote the angle $\sphericalangle ACX_x$, then
         $p_2(x) = 1 - 2\sin(\gamma)\cdot f(x)$ which is concave if and only if $f(x)$ is convex.
          Since $X_xP_x$ is orthogonal to $AX_x$ and $P_x$ is contained in $AC$, we get the following equations on $p_1(x)$ and $p_2(x)$
                     \begin{equation*}
                 p_1(x) \cdot x - p_2(x) = x^2,  \  \  \  \
               p_1(x) + cp_2(x) = c,
             \end{equation*}
            which gives
        $
            p_2(x) 
            = 
            \frac{cx - x^2}{cx+1}.
        $
        Taking the second derivative, we get
        $$
            p''_2(x) = - \frac{2(1+c^2)}{(1+cx)^3} < 0 \text{ for all } x \in \round{-\frac 1 c, \infty}.
        $$
       
\noindent We proved that none of the triangles of types A.1-A.2 and B.1-B.3 is a maximum perimeter embedded isosceles triangle of $ABC$, which completes the proof of \Cref{tfo}(iii). \qed

\section{Minimum perimeter enclosing triangles \\-- Proof of \Cref{thm:main-counterexamples}}\label{sec:min-perim}

In this section, we prove that any smallest perimeter isosceles container of a triangle is either a special container or one of two non-special containers defined in the next subsection. We also show that this is the shortest possible characterization of isosceles containers, that is, any of the five examples 
can be realized as a minimum perimeter isosceles container for some triangle $ABC$.
Now, we define two non-special isosceles containers that can be optimal.

\subsection{Two examples for non-special minimum perimeter containers of a triangle}\label{sec:examples}

Let $P$ be a point in $\RR^2$ and $l$ a line such that $P \not\in l$ and let $m$ denote the distance of $P$ from $l$. 
Define an isosceles triangle $PRS$ such that $S$ and $R$ lie on $l$ and its apex angle $\gamma$ is in $R$, see \Cref{fig:segedlemma}. Let $p$ denote the perimeter of $PRS$. Note that $p$ can be considered as a function of $\gamma$.
\begin{figure}[h]
  \centering
  \includegraphics[width=0.3\textwidth]{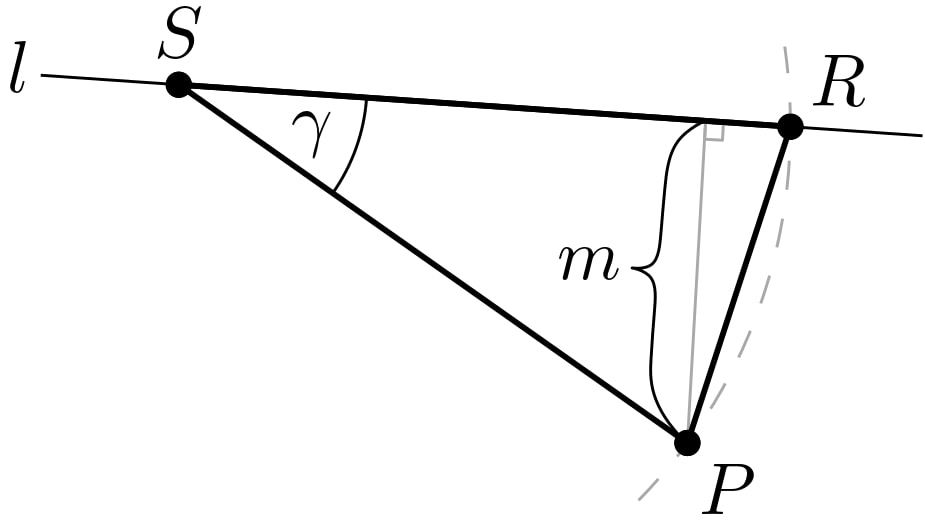}
  \caption{Illustration for \Cref{ltol}}
  \label{fig:segedlemma}
\end{figure}
\begin{prop}\label{ltol}
The function $p$ has a unique minimum at
\begin{equation}\label{eqfig13}
    \gamma^*=4 \tan^{-1}\round{\frac{1}{2}\round{1+\sqrt{5}-\sqrt{2(1+\sqrt 5)}}}  \approx 76.3466^{\circ}. 
    \end{equation}
\end{prop}
\begin{proof}[Proof outline]
It is easy to see that 
$ |PR|= |RS|=\frac{m}{\sin\ga}$ and  $|PS|=\frac{m}{\sin(90^{\circ}-\ga/2)}=\frac{m}{\cos(\ga/2)},$. Hence $\per{PRS}=m\round{\frac{2}{\sin\ga}+\frac{1}{\cos(\ga/2)}}.$
Elementary analysis shows that the function $f(x)=\frac{2}{\sin x}+\frac{1}{\cos (x/2)}$ is strictly decreasing in $(0^{\circ},\gamma^*]$ and strictly increasing in $[\gamma^*, 180^{\circ})$. Thus it has a unique minimum in $0\le x\le 180^{\circ}$ that is taken at the value specified in \Cref{eqfig13}.
\end{proof}
\noindent
\begin{ex}\label{ex0}
Let $PRS$ be an isosceles triangle with apex angle $\ga^*$ which is defined as in \Cref{ltol}.
Now we take an acute triangle $ABC$ in $PRS$ such that $ABC$ and $PRS$ have exactly one common vertex at $A=P$ and $B,C \in SR$ (see Figure \ref{fig:Ex1}). Moreover,  the largest angle $\ga$ of $ABC$ is at $C$ with $\ga<\ga^*$ being close to $\ga^*$ (e.g., $76^{\circ}$) and $ABC$ is almost isosceles 
($|AC|\approx |BC|$).
\end{ex}
\begin{figure}[h!]
  \centering
  \includegraphics[width=0.45\textwidth]{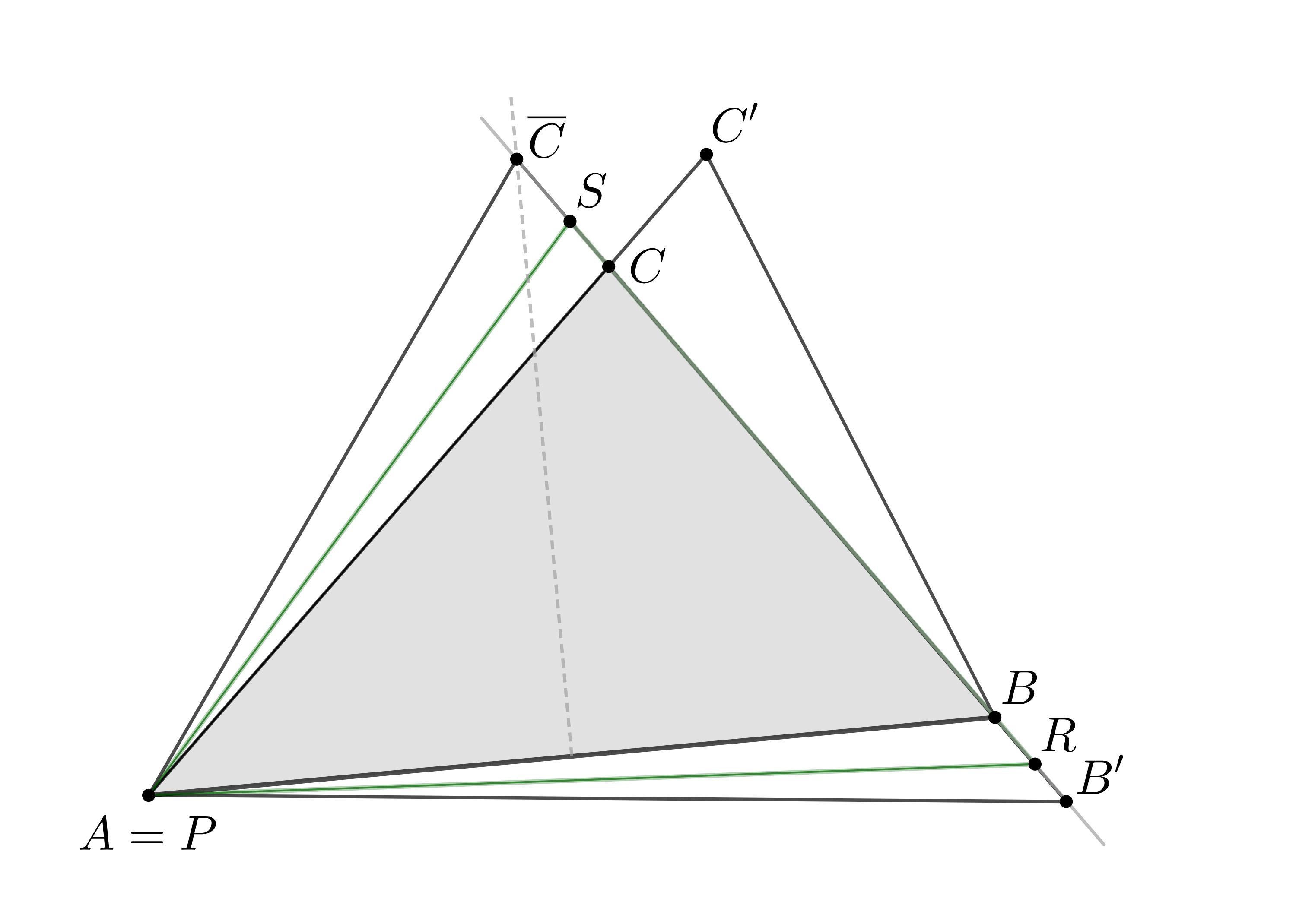}
  \caption{Illustration for Example \ref{ex0}.}
  \label{fig:Ex1}
\end{figure}
\begin{claim}\label{claim:ex1}
The perimeter of $PRS$ is strictly smaller  than  the perimeter of any special container of $ABC$.
 \end{claim}
\begin{proof}[Proof outline]
By \Cref{cspec}, it is enough to show that 
 the special containers $AB'C$, $ABC'$, and $AB\overline{C}$ have larger perimeter than $PRS$. 
First observe that, since $a\approx b<c$, $ABC$ is an `almost' isosceles triangle, thus the perimeter $\per{AB'C}\approx \per{ABC}$ and $\per{ABC'}> d \cdot \per{ABC}$, for a fixed $d>1$. This implies that $\per{AB'C} < \per{ABC'}$.
Now we show that $PRS$ has perimeter smaller than $\per{AB'C}$ and $\per{AB\overline{C}}$. 
Note that, each of $PRS$, $AB'C$ and $AB\overline{C}$ are isosceles triangles with base vertex $A=P$ and legs on the line $RS$.
By \Cref{ltol}, the smallest perimeter isosceles triangle under these conditions is $PRS$.
Thus, it is enough to guarantee that the triangles $AB'C$ and $AB\overline{C}$ do not coincide with $PRS$ which follows from the fact that $ABC$ and $PRS$ has exactly one common vertex.
\end{proof}
\noindent
Now we turn to our second example. We start by taking the points $A=P=(0,0)$, $C=(1,v)$ and $S_x=(x,0)$ and define $R_x$ to be the point on the $\vec{S_xC}$ ray so that $|PR_x|=|R_xS_x|$. The next claim follows by elementary calculations, its proof is omitted.
\begin{prop}\label{claim:ex2}
For any $x \in(1,2)$, the perimeter of $PR_xS_x$ can be expressed as  
\begin{equation}\label{eqex2} \per{PR_xS_x}=f_v(x)= x\round{1+\sqrt{1+\frac{v^2}{(1-x)^2}}}.\end{equation}
and for any $v \in [0.56, \sqrt{3})$, the function $f_v$ has a unique minimum in $(1,2)$ denoted by $x^*_v$.\footnote{The formula for $x_v^*$ is given in the Appendix.}
\end{prop}
\begin{figure}[h!]
  \centering
  \includegraphics[width=0.38\textwidth]{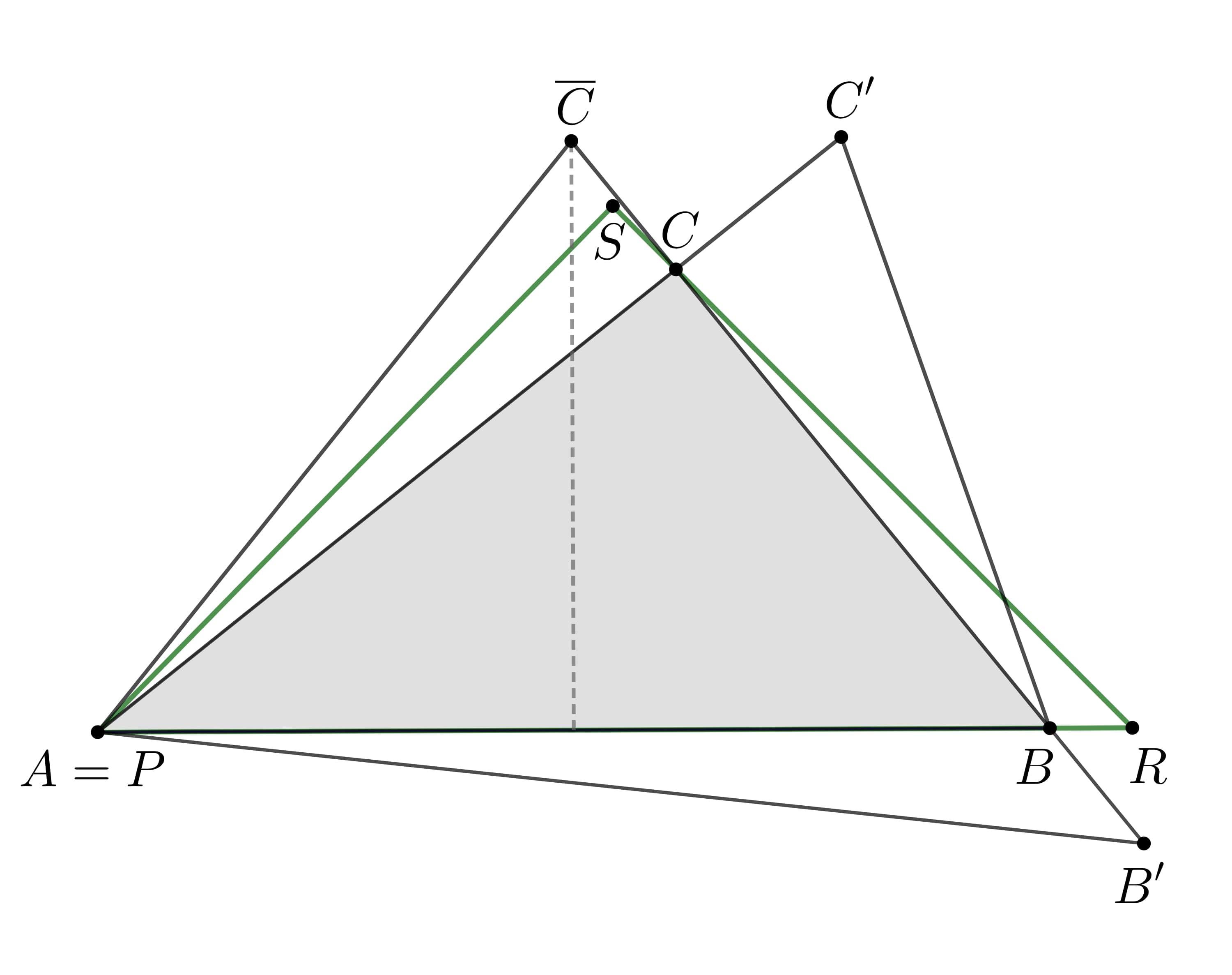}
  \caption{Illustration for Example \ref{ex2}}
  \label{fig_Ex1}
\end{figure}
\begin{ex}\label{ex2}
Consider a triangle $ABC$ that can be embedded in $\RR^2$ as $A=(0,0)$, $C=(1,v)$ and $B=(x_b,0)$ with $ 1 < x_b <  x^*_v$ (the value $x^*_v$ is defined in \Cref{claim:ex2}; see also Figure \ref{fig_Ex1}). Let $PRS$ be the an isosceles triangle with $P=A$, $S = (x^*_v,0) $, and $R$ defined as the point on  the $\vec{S_xC}$ ray with $|PR|=|RS|$. By definition, $SPR$ is an isosceles container of $ABC$.
\end{ex}
\begin{claim}\label{claim:ex2-optimality}
The perimeter of $PRS$ is smaller than the perimeter of any special container of $ABC$.
\end{claim}

\begin{proof}[Proof outline]
By \Cref{cspec}, we only need to show that $PRS$ has a smaller perimeter than the special containers $AB'C, ABC'$, and $AB\overline{C}$.
Observe that by the choices of $x^*_v$ and $x_b$, we have $\per{AB\overline{C}} = f_v(x_b) < f_v(x^*_v) = \per{PRS}$.

We verify the remaining cases only for the fixed value $v=0.7$. The function $f_{0.7}(x)$ takes its minimum at $x_{0.7}^*\approx 1.57517$, and thus $\per{PRS} = f_{0.7}(x_{0.7}^*)\approx 4.056333$. On the other hand, if we set e.g. $x_b = 1.57$, we have $\per{AB'C}\approx 4.229145$ and $\per{ABC'}\approx 4.084007$.
\end{proof}
\subsection{Proof of \Cref{thm:main-counterexamples}}\label{sec:min-perim-proof}
We start by proving that every smallest perimeter isosceles container of a triangle $\Delta=ABC$ is either a special container or one of the two triangles constructed in the Examples \ref{ex0} and \ref{ex2}. Later, we will show that each of these five containers is realized as the unique minimum perimeter isosceles container for some triangle $ABC$.
By Lemmas \ref{lem:01} and \ref{lem:m1}, we have the following statements on minimum perimeter isosceles containers. 
 
 \begin{lem}\label{lemma:GG}
 Let $PRS$ be any minimum area isosceles triangle enclosing $ABC$. Then~\vspace{-.5em}
\begin{enumerate}[(i)]
 \setlength{\itemsep}{1pt}
 \setlength{\parskip}{1pt}
     \item a side of $PRS$ contains a side of $ABC$;
     \item each side of $PRS$ contains a vertex of $ABC$;
     \item $ABC$ and $PRS$ share a common vertex;
     \item no vertex of $ABC$ lies in the interior of $PRS$.
 \end{enumerate}
 \end{lem}

In what follows, $PRS$ is labeled so that  $|PR| = |RS|$.
If  $PRS$ shares the vertex $R$ with $ABC$, but it does not share the angle at $R$, then we can get a smaller perimeter container by decreasing $\sphericalangle SRP$ (while keeping $|PR|=|RS|$ unchanged).
Thus without loss of generality, we can assume that $PRS$ shares the vertex $P$ with $ABC$. The above restrictions allow only the following types of minimum perimeter isosceles containers that do not share an angle with $ABC$ (see also Figure \ref{fig:min-perim-CaseB})\footnote{Note that in this subsection, we do not assume a special labeling of $ABC$, in particular, we do not necessarily have $|BC|<|AC|<|AB|$.}:
    \begin{figure}[h]
    \centering
    \includegraphics[width=0.32\textwidth]{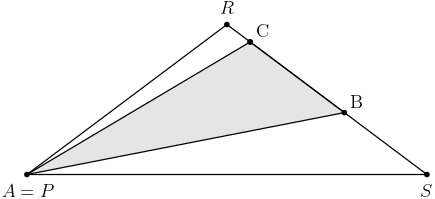}
    \includegraphics[width=0.32\textwidth]{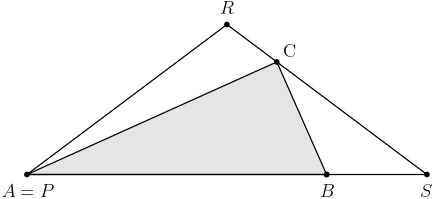}
    \includegraphics[width=0.32\textwidth]{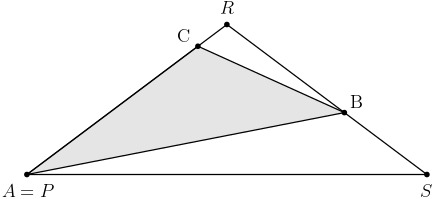}
    \caption{Illustration for Case 1 (left and middle) and Case 2 (right).}
    \label{fig:min-perim-CaseB}
\end{figure} 

\medskip
\noindent
\textbf{Case 1:} \textit{If two vertices of $ABC$ lie in the interior of $RS$, or one of the vertices of $ABC$ lies in the interior of the side $RS$ and one lies in the interior of $PS$. }

The smallest perimeter isosceles containers of these types are precisely the non-special optimal containers shown in Examples \ref{ex0} and \ref{ex2}. 
\medskip

\noindent
\textbf{Case 2:}\textit{ One vertex of $ABC$ is in the interior of $PR$ and one is in the interior of $RS$.}

Let $T$ denote the base of the altitude perpendicular to $RS$ and let $B$ denote the vertex in $RS$. If $|SB|\leq|ST|$ , then $\sa SBP\ge 90^{\circ}$, hence we can rotate $ABC$ with center $A=P$ such that the triangle remains in $PRS$ and hence $PRS$ was not minimal, see \Cref{fig:Case5}. Note that this happens if $PRS$ is not acute. From now on, we assume that $\sa SBP< 90^{\circ}$. Hence $|AB|<|AR|$ if $B\ne R$. 

If $|AC|<|AB|$, then we take $C'\in AR$ such that $|AB|=|AC'|<|AR|$ so $AC'\subset AR$ (\Cref{fig:Case5}). Thus, $ABC'$ is an isosceles container of $ABC$ and $ABC'\subsetneq PRS$. Hence, $PRS$ was not minimal. Therefore, we may assume that $|AC|>|AB|$, as $|AC|=|AB|$ would imply that $ABC$ was isosceles. 
\begin{figure}[h!]
  \centering
  \includegraphics[width=0.6\textwidth]{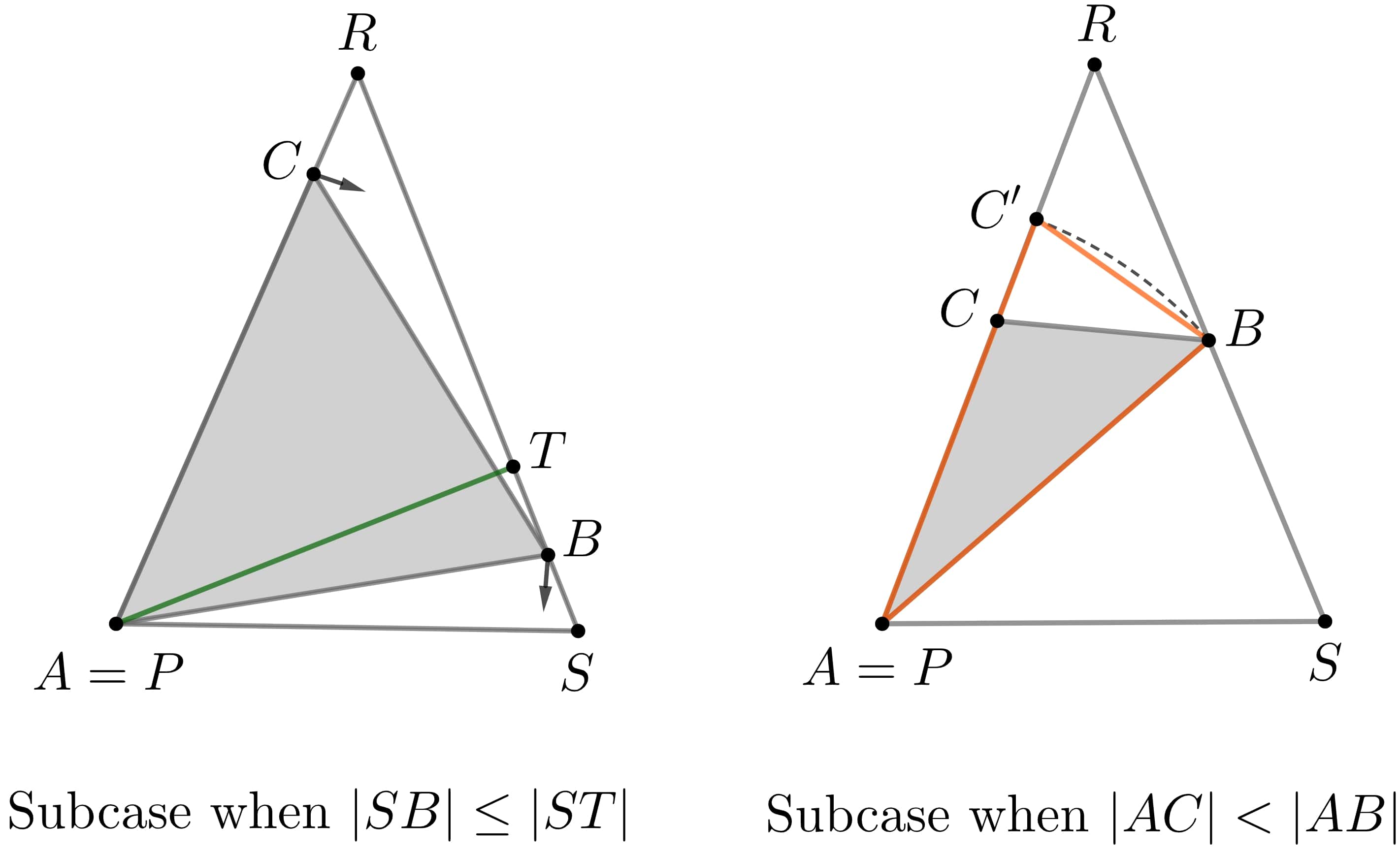}  \caption{Simple configurations of Case 2.}
  \label{fig:Case5}
\end{figure}

If $\sa RAB<\sa BRA$ holds, let $B'$ be the point on the line $AB$ such that $|AC| = |AB'|$ then we have $|AB'|=|AC|<|AR|=|RS|$, and hence $\per{AB'C}<\per{PRS}$, thus $PRS$ was not minimal.
Thus, assume that $ \sa BRA \leq \sa RAB$ and as $\sa RAB+ \sa BRA = \sa SBP < 90^\circ$, we get that $\sa BRA=\sa PRS<45^{\circ}$. 

For the remaining part, we embed the configuration in $\RR^2$ such that
 $P=A=(0,0)$, $R=(x,0)$ and $B=(1,h)$, where $x>1$ and $h>0$, see \Cref{fig:Case5cord}. 
Under the assumptions that $|AC|>|AB|$
and $\sa PRS<\min(\sa RAB, 45^{\circ})$, we show that $\per{PRS}$ as a function of $x$ is increasing. Thus, as $B\ne R, C\ne R$ there exists a smaller perimeter isosceles container of $ABC$ than $PRS$ (e.g. $PR'S'$ in \Cref{fig:Case5cord}). The condition $\sa PRS<\sa RAB$ implies that $|PT'|<|RT'|$, where $T'$ is the base of the altitude of $PR$, hence $x>2$.  
\begin{figure}[h!]
  \centering
  \includegraphics[width=0.5\textwidth]{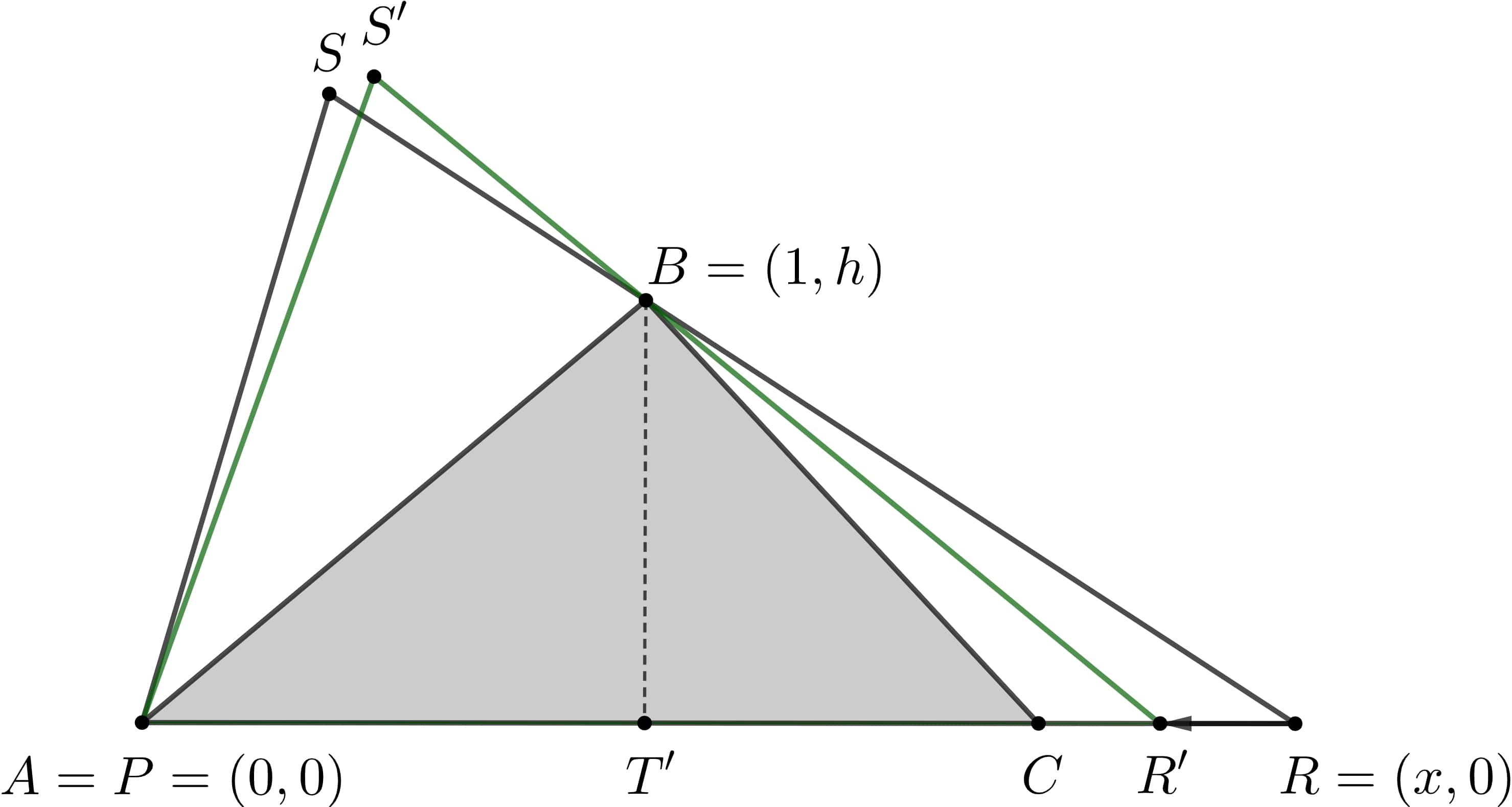}
  \caption{Case 3 in a coordinate system.}
  \label{fig:Case5cord}
\end{figure}

Clearly, $|BR|=\sqrt{h^2+(x-1)^2}$ and $\sin(\sa PRS)=\frac{h}{\sqrt{h^2+(x-1)^2}}$. Hence 
$\per{PRS}=2x(1+\sin(\frac{\sa PRS}{2}))$.
As $\sin\delta=2\sin(\frac{\delta}{2})\sqrt{1-\sin^2(\frac{\delta}{2})}$, we get 
$$\sin\left(\frac{\sa PRS}{2}\right)=\frac{1}{\sqrt{2}}\sqrt{1 \pm \frac{x-1}{h}\sqrt{\frac{1}{1+\left(\frac{x-1}{h}\right)^2}}},$$ where $\pm$ is taken to be $-$ sign, since $\sa PRS<45^{\circ}$. Therefore, 
\[\per{PRS}=2x+\sqrt{2}x\sqrt{1 - \frac{x-1}{h}\sqrt{\frac{1}{1+(\frac{x-1}{h})^2}}}.\]
Let $y=\frac{x-1}{h}$ and let $f_h(y)=(1+hy)\bigg(1+\sqrt{\frac{1 - y}{2}\sqrt{\frac{1}{1+y^2}}}\bigg)$. It follows from our assumptions that $y>1/h$. We show that $f_h(y)$ is strictly increasing in $y$, which implies that $PRS$ is not a minimum perimeter isosceles container of $ABC$. 
For $g(y):=1+\sqrt{\frac{1 - y}{2}\sqrt{\frac{1}{1+y^2}}}$, we show that $f_h'(y)=((1+hy)g(y))'>0$, equivalently $-g'(y)<\frac{hg(y)}{1+hy}$.
Simple calculation shows that $$-g'(y)=\frac{1}{2\sqrt{2} (1+y^2)}\sqrt{1+\sqrt{\frac{y^2}{1+y^2}}}<\frac{1}{2(1+y^2)},$$ where the last inequality holds as $\frac{y^2}{1+y^2}<1$ for all $y\in \mathbb{R}$. 
Note that $g(y)>1$, hence $hg(y)>h$.
Thus, it is enough to show that \begin{equation*}\label{eqveg}
  \frac{1}{2(1+y^2)}<\frac{h}{1+hy} \quad \quad \textrm{ if } y=\frac{x-1}{h}>\frac{1}{h}.
\end{equation*}
This is true if and only if $0<2hy^2-hy+2h-1$. This holds if the roots of this polynomial satisfy $y_1<y_2=\frac{h+\sqrt{-15h^2+8h}}{2h}\le\frac{1}{h}$. The last inequality is equivalent to $0\le 4h^2-3h+1=(2h-1)^2+h$, which is true for $h>0$. Therefore, the argument above verifies that in this case $PRS$ is not minimal.  This concludes the proof in Case 2.
\smallskip
\paragraph{Note on realizability.} Now we briefly discuss that each of the special containers $AB'C$, $ABC'$, $AB\overline{C}$, and triangles constructed in Examples \ref{ex0} and \ref{ex2} can occur as a minimum perimeter container for some $ABC$. It is easy to find triangles for which one of the special containers is the best among the five options.

To see that the container of \Cref{ex2} is optimal for some triangles, note that the construction presented in Example \ref{ex0} works only if the special containers of $ABC$ satisfy $\gamma^*\in(\sa AB'C,\sa AB\overline{C})$.
Now consider the example from the proof of \Cref{claim:ex2-optimality}.
 It can be easily calculated that under these choices $\sa(B\overline{C}A)\approx78,310868 ^{\circ}$. 
This (together with \Cref{claim:ex2-optimality}) implies that for the example presented in the  proof of \Cref{claim:ex2-optimality}, the container described in \Cref{ex2}
is better the one given in \Cref{ex0} and than any special container.

Finally, we show that the container presented in \Cref{ex0} is optimal for some triangles. 
Following the construction in the proof of \Cref{claim:ex2-optimality}, consider the triangle $ABC$ with $A=(0,0),$ and $C=(1,0.8)$ and $B=(0,x^*_{0.8})$ such that $f_{0.8}(x)$ takes its minimum at $x^*_{0.8}\approx1.62474$. We get that the container constructed in \Cref{ex2} coincides with the special container $AB\overline{C}$ and $\per{AB\overline{C}}=f_{0.8}(x^*_{0.8})\approx 4.264511$. Simple calculation shows that $\per{ABC'}\approx 4.3250804$, thus $\per{AB\overline{C}}<\per{ABC'}$. Since $\sa(B\overline{C}A)\approx75.974334^{\circ}<\gamma*<\sa(BCA)=\sa(B'CA)\approx89.327359^\circ$, the construction of \Cref{ex0} provides smaller perimeter than any of the special containers, indeed if we let $SPR$ to be the container constructed in \Cref{ex0} for out choice of $ABC$, then we get $\per{PRS}\approx 4.264431$.

This concludes the proof of \Cref{thm:main-counterexamples}. \qed

\bibliographystyle{siam}


\section*{Appendix - Basic inequalities for embedded triangles}
\subsubsection*{Proof of  \Cref{lem:01}}
Observe that if $\Delta$ and $ \Delta'$ are similar triangles then $\per{\Delta}<\per{\Delta'}$ and $\ar{\Delta}<\ar{\Delta'}$ hold if and only if $\diam(\Delta)<\diam(\Delta')$.
Therefore, it is sufficient to prove following the statements:
\begin{claim}\label{claim:rephrase}
Let $\Delta_1$ be a triangle inside the triangle $\Delta_2$. 
If $\Delta_1'$ is a maximum diameter triangle similar to $\Delta_1$ contained in $\Delta_2$, then~\vspace{-.5em}
\begin{enumerate}
 \setlength{\itemsep}{1pt}
  \setlength{\parskip}{1pt}
    \item[(A)] $\Delta_1'$ has two vertices on a side of $\Delta_2$;
    \item[(B)] each side of $\Delta_2$ contains a vertex of $\Delta_1'$;
    \item[(C)] $\Delta_1'$ and $\Delta_2$ have a common vertex.
\end{enumerate}
\end{claim}

\medskip\noindent
\textit{Proof.}
Let $\Delta_1 = ABC$, $\Delta'_1 = A'B'C'$, and $\Delta_2 = DEF$.~\vspace{-.5em}
\begin{enumerate}
 \setlength{\itemsep}{1pt}
  \setlength{\parskip}{1pt}
\item[(A)] The statement was proved by Post~\cite{Post}.
\item[(B)] Assume indirectly that the side $DE$ does not contain any of the vertices $A',B',C'$ and let $A'$ the closest one to $DE$.
  Let $l'$ denote the line parallel to $DE$ which contains $A'$. 
Let $D'=l'\cap DF$ and $E'=l'\cap EF$. Then the triangle $D'E'F\subset DEF$ is similar to $DEF$ and it contains $A'B'C'$. 
If $\mathcal S$ is the homothety with center $F$ and ratio $|D F| / |D'F|$, then $\mathcal{S}(A'B'C')\subset DEF$ is a similar copy of $A'B'C'$ which has larger diameter than $A'B'C'$, a contradiction. Hence, $DE$ contains a vertex of $A'B'C'$.
The argument applies for any other side of $DEF$. 
\item [(C)] By part (A), $A'B'C'$ has two vertices that are on the same side of $DEF$. If any of these vertices coincide with a vertex of $DEF$, then the second part of the statement holds. Otherwise,
the third vertex of $A'B'C'$ must be contained in two sides of $DEF$, thus it is a 
vertex of $DEF$.   \qed
\end{enumerate}
\subsubsection*{Proof of \Cref{lem:m1}} Let $\Delta = ABC$, $\Delta_1=DEF,$ and $\Delta_2 = PRS$.
\begin{enumerate}
\item[(i)]
We prove that the vertices of minimum area (resp. perimeter) isosceles containers are contained in a compact subset of the plane. (A similar result for the set of embedded triangles is trivial.)

The case of the area was handled in \cite{KPS2020}. Concerning the perimeter, it is easy to see that there is a special container whose perimeter is at most twice as large as the one of $ABC$. Thus  any minimum perimeter container of $ABC$ is contained in the closed ball $\mathcal{B}$ centered at $A$ with radius $2\per{ABC}$. 

Thus the collection of those isosceles triangles $\overline{\Delta}$ satisfies $ABC\subset \overline{\Delta}\subset \mathcal{B}$ can be considered as a compact subset of $\mathbb R^6$ with respect to the Euclidean topology. 
The result simply follows from the fact that both the area and the perimeter are continuous functions on the parameter set.

\item[(iii)] The statement for the minimum area has been proved in \cite{KPS2020}[Lemma 3.2]. 

Let $PRS$ be a minimum perimeter isosceles container of $ABC$. Then, by \Cref{claim:rephrase}, one side of $PRS$ contains two vertices of $ABC$, and by \Cref{lem:01} every side of $PRS$ contains a vertex of $ABC$ and the triangles share a common vertex. Thus either each vertex of $ABC$ is on a side of $PRS$, as we stated, or two vertices of $ABC$ coincide with two vertices of $PRS$ (i.e., the triangles share a common side) as in \Cref{fig:shrink}. In the latter case we distinguish two subcases when the common side is a leg (Case a.) or a base (Case b.) of $PRS$. As \Cref{fig:shrink} illustrates, in both subcases we can find a smaller isosceles triangle (green in \Cref{fig:shrink}) by shrinking the original triangle $PRS$ so that the modified isosceles triangle contains $ABC$ and has smaller area and perimeter.   
\begin{figure}[h!]
\centering
  \centering
  \includegraphics[width=0.8\textwidth]{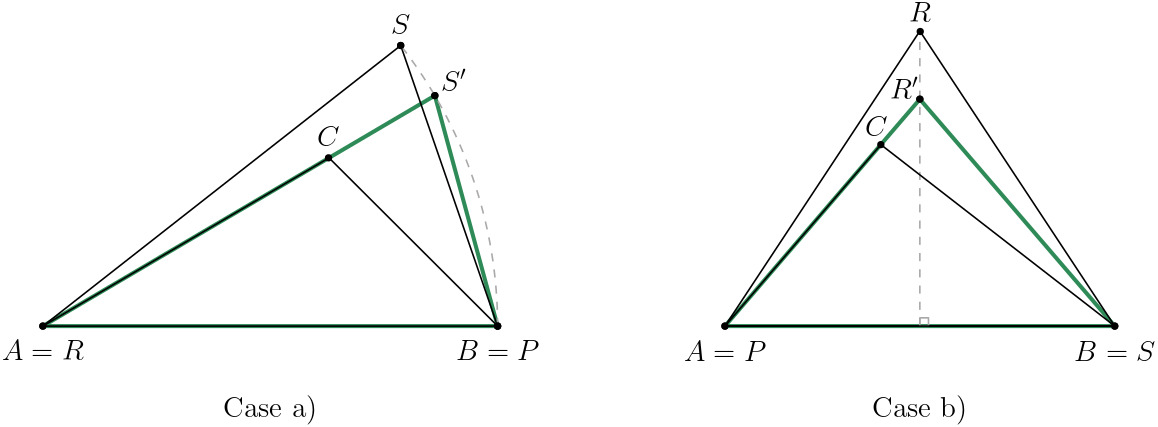}
  \caption{Illustration for the proof of  \Cref{lem:m1} (iii).}
  \label{fig:shrink}
\end{figure}
\item[(ii)] The proof of this case is analogous to the proof of case \eqref{lemitem:242} for the perimeter and the proof of \cite{KPS2020}[Lemma 3.2] for the area. 
\end{enumerate}
\qed
\subsubsection*{Proof of \Cref{lem:ie2}}
The areas of special embedded triangles of the first, the second and the third type are the following:  $$\ar{A'BC}=\frac{a^2\sin\gamma}{2}, \ \ \ar{AB'C}=\frac{b^2\sin\alpha}{2}, \ \ \ar{A''BC}=\frac{a^2\sin\beta}{2}.$$
$$\ar{A_1BC}=\frac{a^2\tan\beta}{4}, \ \ \ar{AB_1C}=\frac{b^2\tan\alpha}{4},\ \ \ar{ABC_1}=\frac{c^2\tan\alpha}{4}.$$
$$\ar{\overline{A}BC}=\frac{a^2\sin(2\beta)}{2}, \ \ \ar{\overline{\overline{A}}BC}=\frac{a^2\sin(2\gamma)}{2},  \ \ \ar{A\overline{B}C}=\frac{b^2\sin(2\gamma)}{2}.$$
\begin{enumerate}
    \item It follows easily from the previous equalities that  $$\ar{A''BC}=\frac{a^2\sin\beta}{2} < \frac{a^2\sin\beta}{2} \frac{c}{b} = \frac{a^2\sin\beta}{2} \frac{\sin\gamma}{\sin\beta}=\ar{A'BC}.$$ 
    \item 
    We have seen that $\ar{A_1BC}=\frac{a^2\tan\beta}{4}$ and $\ar{AB'C}=\frac{b^2\sin\alpha}{2}$. Now $\frac{a^2\tan\beta}{4} < \frac{b^2\sin\alpha}{2}$ is equivalent to $\frac{a^2}{2b^2}=\frac{\sin^2\alpha}{2\sin^2\beta} < \frac{\sin\alpha\cos\beta}{\sin\beta}$, by using the law of sines. Reformulating this we have $\sin\alpha<\sin(2\beta)$, which always holds as $\alpha < \min \{ 2 \beta, 180- 2 \beta \}$.
    
    The inequality $\frac{b^2\tan\alpha}{4} < \frac{c^2\tan\alpha}{4}$ implies that $\ar{AB_1C}< \ar{ABC_1}$. 
\item 
We first prove
$$\ar{\overline{A}BC} < \ar{ABC_1}.$$ Using the equations given above this is equivalent to
$$\frac{a^2\sin(2\beta)}{2} < \frac{c^2\tan\alpha}{4}\Longleftrightarrow a^2\sin\beta\cos\beta < \frac{c^2\sin\alpha}{4\cos\alpha}.$$
By replacing $\frac{\sin\alpha}{\sin\beta}$ with $\frac{a}{b}$ we obtain $ \frac{4ab}{c^2} \cos\alpha\cos\beta<1.$ 
Law of cosines gives 
$$ \frac{4ab}{c^2} \frac{2bc}{b^2+c^2-a^2}\frac{2ac}{a^2+c^2-b^2}<1.$$
After a simple rearrangement we obtain 
$$(c^2+(a^2-b^2))(c^2-(a^2-b^2)) < c^4, $$
which holds by $a \ne b$.

The inequality $\ar{A\overline{B}C} > \ar{\overline{\overline{A}}BC}$ follows from
 $$\ar{\overline{\overline{A}}BC}=\frac{a^2\sin(2\gamma)}{2} < \frac{b^2\sin(2\gamma)}{2}=\ar{A\overline{B}C}.$$

The length of the legs of the isosceles triangles $A\overline{B}C$ and $AB'C$ are equal to $b$, but the apex angle is greater in the latter triangle. As both apex angles are upper bounded by  $\alpha < 90^\circ$, we get
 $\ar{A\overline{B}C} < \ar{AB'C}$. 
\item 
If $ABC$ is obtuse, then $0^\circ<\alpha< 45^\circ$. The function $\sin(2\alpha)$ is strictly monotonically increasing on the interval $[0^\circ 45^\circ]$. 
Since $\alpha < \beta$ we have $\alpha <180^\circ-\gamma <90^{\circ}$. Thus
$$\sin(2\alpha) < \sin(180-\gamma)= \sin\gamma,$$
$$2\sin\alpha\cos\alpha < \sin\gamma$$
$$\sin\alpha  < \frac{\sin\gamma}{2\cos\alpha}.$$
We multiply both sides of the inequality with the positive number  $\frac{\sin\alpha\sin\gamma}{2}$:
$$\frac{\sin^2\alpha\sin\gamma}{2} < \frac{\sin\alpha\sin^2\gamma}{4\cos\alpha}.$$
By the law of sines this is equivalent to
$$\frac{a^2\sin\gamma}{2} < \frac{c^2\sin\alpha}{4\cos\alpha},$$
which implies
\begin{equation*}
  \ar{A'BC}= \frac{a^2\sin\gamma}{2}\leq \frac{c^2\tan\alpha}{4}=\ar{ABC_1}.
\end{equation*}

\qed
\end{enumerate}

\subsubsection*{Proof of \Cref{lem:ie1}}
\begin{enumerate}
  \item 
  $\per{ABC'}<\per{ABC''}$ follows from the Hinge theorem since
   $|AB|=|AC'|=|AC''|=c$ and $\alpha=\sa C'AB<\sa C''BA=\beta<90^{\circ}$. 
   \item Notice that $|AB|=|AC_1|=|AC_2|=c$. On the other hand $\sa BAC'=\alpha < \sa BAC_2=180^{\circ}-2 \beta < 180^{\circ}-2 \alpha$ since $\alpha+\beta+\gamma=180^{\circ}$ and $\alpha < \beta <\gamma$. Thus the Hinge theorem gives $\per{ABC'}<\per{ABC_2} < \per{ABC_1}$. 
\item 
Similarly to the previous cases, we have $|AC|=|CB'|=|CB_1|=b$ and $\sa ACB' < \sa ACB_1$ so the perimeter of $AB'C$ is smaller than that of $AB_1C$.

  \item First note that the triangles $\overline{A}BC$ and $A\overline{B}C$ do exist if and only if $ABC$ is acute, hence we assume $\gamma<90^{\circ}$. First we show that $\per{ABC'}>\per{\overline{A}BC}$. Indeed, 
  $\per{ABC'}=2c(1+\sin(\alpha/2))$ and
  $\per{\overline{A}BC}=a(1+1/\cos\gamma)$, thus it is enough to show that 
 $$2c(1+\sin(\alpha/2)) < a(1+1/\cos\gamma).$$
Using the law of sines we obtain
 $$\frac{2(1+\sin(\alpha/2))}{1+1/\cos\gamma} < \frac{a}{c}= \frac{\sin\alpha}{\sin\gamma}.$$
 Equivalently, 
$$ \frac{2(1+\sin(\alpha/2))}{\sin\alpha} < \frac{1+1/\cos\gamma} {\sin\gamma}= \frac{2(1+ \cos\gamma)}{\sin (2\gamma)}=\frac{2(1+ \sin(90^{\circ}-\gamma))}{\sin(180^{\circ}-2\gamma)}.$$
Note that $(1+\sin x)/\sin (2x)$ is strictly decreasing on the interval $(0^{\circ},60^{\circ})$. It is clear that $90^{\circ}-\gamma<\alpha$, otherwise $90^{\circ}<\beta<\gamma$, contradicting the assumption that $\gamma<90^{\circ}$.    
Hence we get that $\per{ABC'} < \per{\overline{A}BC}$.

The inequality $\per{\overline{A}BC}<\per{A\overline{B}C}$ simply follows from that fact that $\overline{A}BC$ and $A\overline{B}C$ are similar triangles such that the base of $\overline{A}BC$, is of length $a$ so it shorter than the base of $A\overline{B}C$, which is of length $b$. 
\end{enumerate}
\hfill $\square$

\subsection*{Formula for $x_v^*$}

Let $\delta_v :=\sqrt{48v^6+81v^4} - 9v^2$, then $x_v^*$ can be expressed as

\begin{align*}
    x_v^* 
    =
    \frac 1 2\round{ 1 +  \sqrt{
    1 
    + 
    \sqrt[3]{\frac {2\delta_v} 9 }
    -
    \sqrt[3]{\frac{2^5v^6}{3\delta_v}}
    }
    +
     \sqrt{
    2
    +
    \sqrt[3]{\frac{2^5v^6}{3\delta_v}}
    -
     \sqrt[3]{\frac {2\delta_v} 9}
    -\frac{2}{
    \sqrt{
    1 +
     \sqrt[3]{\frac {2\delta_v} 9}
    -
     \sqrt[3]{\frac{2^5v^6}{3\delta_v}}
    }}
    }
    }.
\end{align*}

\end{document}